\documentclass{amsart}
\usepackage[all]{xy}
\usepackage{enumerate}
\usepackage{mathrsfs}
\usepackage{amssymb}
\usepackage{graphicx}

\numberwithin{equation}{subsection}

\newcommand{\nc}{\newcommand}
\nc\rnc{\renewcommand}

\theoremstyle{plain}
\newtheorem{theorem}{Theorem}[section]
\newtheorem{lemma}[theorem]{Lemma}
\newtheorem{corollary}[theorem]{Corollary}
\newtheorem{proposition}[theorem]{Proposition}
\newtheorem{conjecture}[theorem]{Conjecture}
\theoremstyle{definition}
\newtheorem{definition}[theorem]{Definition}

\theoremstyle{remark}
\newtheorem{remark}[theorem]{Remark}
\newtheorem{question}[theorem]{Question}

\newcommand\Aut{\operatorname{Aut}}
\newcommand\Out{\operatorname{Out}}

\newcommand\Hom{\operatorname{Hom}}

\newcommand\GL{\operatorname{GL}}
\newcommand\SL{\operatorname{SL}}
\newcommand\IA{\operatorname{IA}}
\newcommand\IO{\operatorname{IO}}

\newcommand\id{\operatorname{id}}

\newcommand\im{\operatorname{im}}

\newcommand\gr{\operatorname{gr}}

\newcommand\ab{\operatorname{ab}}
\newcommand\sgn{\operatorname{sgn}}

\newcommand\Z{\mathbb{Z}}
\newcommand\N{\mathbb{N}}

\newcommand\R{\mathbb{R}}
\newcommand\Q{\mathbb{Q}}

\newcommand\gpS{\mathfrak{S}}

\newcommand\wti{\widetilde}

\newcommand\red[1]{{\color{red}#1}}
\renewcommand\red{}

\newcommand{\GLnZ}{\GL(n,\Z)}
\newcommand{\SLnZ}{\SL(n,\Z)}

\nc\FAb{\mathbf{FAb}}

\nc\xto[1]{{\overset{#1}{\longrightarrow}}}
\nc\yto[1]{{\underset{#1}{\longrightarrow}}}
\nc\xyto[2]{{\overset{#1}{\underset{#2}{\longrightarrow}}}}
\nc\ad{{\operatorname{ad}}}
\nc\even{{\operatorname{even}}}
\nc\ev{{\operatorname{ev}}}
\nc\coev{{\operatorname{coev}}}
\nc\odd{{\operatorname{odd}}}
\nc\half{{\frac12}}
\nc\halfof[1]{{\frac{#1}2}}
\nc\onto\twoheadrightarrow
\nc\projto{\underset{\text{proj}}{\longrightarrow}}
\nc\no[1]{}
\nc\ok{\comm{ok?}}
\nc\ho{{\hat\otimes }}
\nc\plim{\varprojlim}
\nc\np{\newpage}

\nc\bfA{\mathbf{A}} \nc\bbA{\mathbb{A}} \nc\calA{\mathcal{A}}
\nc\bfB{\mathbf{B}} \nc\bbB{\mathbb{B}} \nc\calB{\mathcal{B}}
\nc\bfC{\mathbf{C}} \nc\bbC{\mathbb{C}} \nc\calC{\mathcal{C}}
\nc\bfD{\mathbf{D}} \nc\bbD{\mathbb{D}} \nc\calD{\mathcal{D}}
\nc\bfE{\mathbf{E}} \nc\bbE{\mathbb{E}} \nc\calE{\mathcal{E}}
\nc\bfF{\mathbf{F}} \nc\bbF{\mathbb{F}} \nc\calF{\mathcal{F}}
\nc\bfG{\mathbf{G}} \nc\bbG{\mathbb{G}} \nc\calG{\mathcal{G}}
\nc\bfH{\mathbf{H}} \nc\bbH{\mathbb{H}} \nc\calH{\mathcal{H}}
\nc\bfI{\mathbf{I}} \nc\bbI{\mathbb{I}} \nc\calI{\mathcal{I}}
\nc\bfJ{\mathbf{J}} \nc\bbJ{\mathbb{J}} \nc\calJ{\mathcal{J}}
\nc\bfK{\mathbf{K}} \nc\bbK{\mathbb{K}} \nc\calK{\mathcal{K}}
\nc\bfL{\mathbf{L}} \nc\bbL{\mathbb{L}} \nc\calL{\mathcal{L}}
\nc\bfM{\mathbf{M}} \nc\bbM{\mathbb{M}} \nc\calM{\mathcal{M}}
\nc\bfN{\mathbf{N}} \nc\bbN{\mathbb{N}} \nc\calN{\mathcal{N}}
\nc\bfO{\mathbf{O}} \nc\bbO{\mathbb{O}} \nc\calO{\mathcal{O}}
\nc\bfP{\mathbf{P}} \nc\bbP{\mathbb{P}} \nc\calP{\mathcal{P}}
\nc\bfQ{\mathbf{Q}} \nc\bbQ{\mathbb{Q}} \nc\calQ{\mathcal{Q}}
\nc\bfR{\mathbf{R}} \nc\bbR{\mathbb{R}} \nc\calR{\mathcal{R}}
\nc\bfS{\mathbf{S}} \nc\bbS{\mathbb{S}} \nc\calS{\mathcal{S}}
\nc\bfT{\mathbf{T}} \nc\bbT{\mathbb{T}} \nc\calT{\mathcal{T}}
\nc\bfU{\mathbf{U}} \nc\bbU{\mathbb{U}} \nc\calU{\mathcal{U}}
\nc\bfV{\mathbf{V}} \nc\bbV{\mathbb{V}} \nc\calV{\mathcal{V}}
\nc\bfW{\mathbf{W}} \nc\bbW{\mathbb{W}} \nc\calW{\mathcal{W}}
\nc\bfX{\mathbf{X}} \nc\bbX{\mathbb{X}} \nc\calX{\mathcal{X}}
\nc\bfY{\mathbf{Y}} \nc\bbY{\mathbb{Y}} \nc\calY{\mathcal{Y}}
\nc\bfZ{\mathbf{Z}} \nc\bbZ{\mathbb{Z}} \nc\calZ{\mathcal{Z}}
\nc\bfone {{\mathbf 1}}

\nc\Vect{\mathbf{Vect}}
\nc\Sets{\mathbf{Sets}}
\nc\Mod{\mathbf{Mod}}
\nc\Cat{\mathbf{Cat}}

\nc\ul{\underline}
\nc\simeqto{\overset{\simeq}{\longrightarrow }}
\nc\ct{\overset{\cong}{\longrightarrow }}
\nc\mt{\mapsto}
\nc\hr{\medskip\hrule\medskip}
\nc\trl{\triangleleft}
\nc\trr{\triangleright}

\nc\xysquare[8]{\xymatrix{
    #1 \ar[r]#5 \ar[d]#6 & #2 \ar[d]#7 \\
    #3 \ar[r]#8          & #4
  }
  }

\nc\Ob{\operatorname{Ob}}
\nc\Mor{\operatorname{Mor}}

\nc\al{\alpha}
\nc\be{\beta}
\nc\la{\lambda}
\nc\ot{\otimes}
\nc\ott{\ot\cdots\ot}
\nc\Sp{\operatorname{Sp}}
\nc\SO{\operatorname{SO}}

\nc\HH{\mathrm{HH}}

\nc\ci{\circ}
\nc\sq{\square}

\nc\incl{\mathrm{incl}}

\nc\ol{\overline}
\nc\sqcups{\sqcup\cdots\sqcup}
\nc\congto{\overset{\cong}{\to}}

\nc\hide[1]{}
\nc\blue[1]{{\textcolor[rgb]{0,0,.9}{#1}}}
\nc\bluen[1]{\blue{[[#1]]}}
\nc\bnote{\bluen}
\nc\redn[1]{\red{[[#1]]}}

\nc\comment[1]{\marginpar{\tiny #1}}
\nc\len{\operatorname{len}}

\nc\VIC{\mathrm{VIC}}

\nc\GLnQ{\GL(n,\Q)}

\nc\nKMP{n_{\mathrm{KMP}}}
\nc\nB{n_{\mathrm{B}}}

\nc\Pol{\mathcal{P}ol}
\nc\VICmod{\VIC\text{-}\mathrm{mod}}
\nc\colim{\operatorname{colim}}

\nc\ulla{{\ul\lambda}}
\nc\Sz{\mathsf{Sz}}
\nc\St{\mathsf{St}}

\nc\SF{\mathrm{SF}}
\nc\AutFn{\Aut(F_n)}

\let\copybigwedge\bigwedge
\renewcommand\bigwedge{\copybigwedge\nolimits}

\title[On the stable cohomology of ($\IA$-)automorphism groups]{On the stable cohomology of the $\IA$-automorphism groups of free groups}

\author{Kazuo Habiro}
\author{Mai Katada}
\address{Department of Mathematics, Kyoto University, Kyoto 606-8502, Japan}
\curraddr{Graduate School of Mathematical Sciences, University of Tokyo, Tokyo 153-8914, Japan}
\email{habiro@ms.u-tokyo.ac.jp}
\address{Department of Mathematics, Kyoto University, Kyoto 606-8502, Japan}
\curraddr{Faculty of Mathematics, Kyushu University, Fukuoka 819-0395, Japan}
\email{katada@math.kyushu-u.ac.jp}

\date{January 17, 2024 (First version: November 24, 2022)}


\begin{document}

\maketitle

\begin{abstract}
Borel's stability and vanishing theorem gives the stable cohomology of $\GL(n,\Z)$ with coefficients in algebraic $\GL(n,\Z)$-representations.
By combining the Borel theorem with the Hochschild--Serre spectral sequence, we compute the twisted first cohomology of the automorphism group $\Aut(F_n)$ of the free group $F_n$ of rank $n$. We also study the stable rational cohomology of the IA-automorphism group $\IA_n$ of $F_n$. We propose a conjectural algebraic structure of the stable rational cohomology of $\IA_n$, and consider some relations to known results and conjectures. We also consider a conjectural structure of the stable rational cohomology of the Torelli groups of surfaces.
\end{abstract}

\setcounter{tocdepth}{1}
\tableofcontents

\section{Introduction}

Let $F_n$ be the free group of rank $n$, $\Aut(F_n)$ the automorphism group of $F_n$ and $\GLnZ$ the general linear group of rank $n$.
Let $\IA_n$ denote the IA-automorphism group of $F_n$, which is defined by the following exact sequence
\begin{gather}\label{exact}
 1\to \IA_n\to \Aut(F_n)\to \GLnZ\to 1.
\end{gather}
The Hochschild--Serre spectral sequence for the above exact sequence
gives a strong connection between the twisted cohomology of $\Aut(F_n)$, the twisted cohomology of $\GLnZ$ and the rational cohomology of $\IA_n$.

Borel \cite{Borel1} proved that the rational cohomology of $\GLnZ$ \emph{stabilizes}, which means that the map on the rational cohomology in each degree induced by the inclusion $\GLnZ\hookrightarrow \GL(n+1,\Z)$ is an isomorphism for sufficiently large $n$, and he computed the stable rational cohomology of $\GLnZ$.
Borel \cite{Borel2} also computed the stable cohomology of $\GLnZ$ with coefficients in algebraic $\GLnZ$-representations.
Li and Sun \cite{Li-Sun} obtained an improved stable range that is independent of the types of algebraic $\GL(n,\Z)$-representations.

Hatcher and Vogtmann \cite{Hatcher-Vogtmann98,Hatcher-Vogtmann} showed that the rational homology of $\Aut(F_n)$ stabilizes, and Galatius \cite{Galatius} showed that the stable rational homology is trivial.
Let $H_{\Z}=H_1(F_n,\Z)$, $H_{\Z}^*=\Hom_{\Z}(H,\Z)$, $H=H_{\Z}\otimes \Q$ and $H^*=\Hom(H,\Q)$.
Satoh \cite{Satoh2006,Satoh2007, Satoh2013} computed the low-degree (co)homology with coefficients in $H_{\Z}, H_{\Z}^*$ and $H_{\Z}^*\otimes \bigwedge^2 H_{\Z}$.
The stable homology with coefficients in $H_{\Z}^{\otimes p}$ and $(H^*)^{\otimes q}$ was computed by Djament and Vespa \cite{Djament-Vespa}, Vespa \cite{Vespa} and
Randal-Williams \cite{Randal-Williams}.
Kawazumi and Vespa \cite{Kawazumi-Vespa} proposed a conjectural structure of the stable cohomology of $\Aut(F_n)$ with coefficients in $H^{\otimes p}\otimes (H^*)^{\otimes q}$,
which has been proven by Lindell \cite{Lindell22} recently.

By the exact sequence \eqref{exact}, we have a $\GL(n,\Z)$-action on the rational cohomology of $\IA_n$.
Cohen and Pakianathan \cite{CP}, Farb \cite{Farb} and Kawazumi \cite{Kawazumi} independently determined the first cohomology of $\IA_n$.
Pettet \cite{Pettet} and the second-named author \cite{KatadaIA} determined the second and third Albanese cohomology of $\IA_n$, which is the $\GL(n,\Z)$-subrepresentation of the rational cohomology of $\IA_n$ obtained as the image of the rational cohomology of the abelianization of $\IA_n$ under the induced map by the abelianization map.
Day and Putman \cite{Day-Putman} showed that the $\GLnZ$-invariant part of $H^2(\IA_n,\Z)$ vanishes.

Unlike the case of $\GLnZ$ and $\Aut(F_n)$, the rational cohomology of $\IA_n$ does \emph{not} stabilize.
In what follows, by the \emph{stable rational cohomology of $\IA_n$} we mean the rational cohomology of $\IA_n$ for sufficiently large $n$.
The rational cohomology of $\IA_n$ is expected to stabilize as a family of $\GL(n,\Z)$-representations, which is called \emph{representation stability} \cite{Church-Farb}.
In fact, the first rational cohomology of $\IA_n$ is representation stable.
The stable rational cohomology of $\IA_n$ in degree greater than $1$ has not been determined.
The second-named author \cite{KatadaIA} proposed a conjectural structure of the stable Albanese cohomology of $\IA_n$.
To the best of our knowledge, no conjecture about the structure of the whole stable rational cohomology of $\IA_n$ has been given in the literature.

In this paper, we consider the cohomology in a stable range of the families of groups $\Aut(F_n)$ and $\IA_n$ with coefficients in algebraic $\GL(n,\Z)$-representations.
By using the improved Borel theorem and the Hochschild--Serre spectral sequence associated to the exact sequence \eqref{exact}, we compute the first cohomology of $\Aut(F_n)$ with coefficients in algebraic $\GLnZ$-representations and the $\Aut(F_n)$-modules of \emph{Jacobi diagrams} and their duals.
We propose a conjectural algebraic structure of the stable rational cohomology of $\IA_n$.
We also consider the stable rational cohomology of $\IO_n$, which is an analogue of $\IA_n$ to the outer automorphism group $\Out(F_n)$ of $F_n$, and the stable rational cohomology of the Torelli groups of surfaces.

Here we make some conventions about stability.
We say that some statement about cohomology of groups depending on $n$ holds \emph{stably} if it holds for $n$ sufficiently large with respect to cohomological degree.
A \emph{stable isomorphism} means a morphism that is stably an isomorphism.
Note that if we say ``we stably have an isomorphism'', then we do not assume a morphism in an unstable range, while if we say ``we have a stable isomorphism'', then we assume that we have a morphism in both stable and unstable ranges.

\subsection{First cohomology of $\Aut(F_n)$ with coefficients in algebraic $\GL(n,\Z)$-representations}

Let $H=H_1(F_n,\Q)$ and $H^*=H^1(F_n,\Q)=\Hom_{\Q}(H,\Q)$.
Let $V_{\ul\la}$ denote the irreducible algebraic $\GL(n,\Z)$-representation corresponding to the bipartition $\ul\la$, i.e., a pair of partitions. See Section \ref{alg-glnz-rep} for algebraic $\GL(n,\Z)$-representations. 

It is well known that we have $H^1(\Aut(F_n),\Q)=0$. We consider the twisted first cohomology of $\Aut(F_n)$. We have the following result of Satoh \cite{Satoh2013, Satoh2006}.

\begin{theorem}[Satoh \cite{Satoh2013, Satoh2006}]
We have 
  \begin{gather*}
     H^1(\Aut(F_n),H^*)=0, \quad H^1(\Aut(F_n),H)=\Q \quad \text{for}\; n\ge 2,\\
    H^1(\Aut(F_n),V_{1^2,1})=\Q \quad \text{for}\; n\ge 5.
  \end{gather*}
\end{theorem}

Djament and Vespa \cite{Djament-Vespa} showed the vanishing of the stable homology of $\Aut(F_n)$ with coefficients induced by reduced polynomial functors on the category of finitely generated free groups. In particular, the stable homology with coefficients in $H^{\otimes p}$ vanishes for any $p\ge 1$.
By using Djament's result \cite{Djament}, Vespa \cite{Vespa} obtained the stable homology of $\Aut(F_n)$ with coefficients in $(H^*)^{\otimes q}$ for any $q\ge 1$.
Randal-Williams \cite[Theorem A]{Randal-Williams} obtained the following stable range by using geometric techniques
\begin{gather*}
 \begin{split}
     H^1(\Aut(F_n), (H^*)^{\otimes q})&=0,\quad \text{for}\; n\ge q+5,\\
    H^1(\Aut(F_n), H^{\otimes q})&=0,\quad \text{for}\; n\ge q+5,\; q\ne 1.
 \end{split}
\end{gather*}

We obtain the vanishing of the first cohomology of $\Aut(F_n)$ with coefficients in algebraic $\GL(n,\Z)$-representations $V_{\ulla}$ for any bipartitions $\ulla\neq (1,0),(1^2,1)$.

\begin{theorem}[Theorem \ref{H1Aut}]
For any bipartition $\ul\lambda\neq (1,0),(1^2,1)$, we have
$$  
 H^1(\Aut(F_n),V_{\ul\lambda})=0 \quad\text{for}\; n\ge 3.  
$$
For $\ul\lambda= (1,0),(1^2,1)$, 
  \begin{gather*}
    \begin{split}
        H^1(\Aut(F_n),V_{1,0})&=\Q \quad\text{for}\; n\ge 4,\\
        H^1(\Aut(F_n),V_{1^2,1})&=\Q \quad\text{for}\; n\ge 4.
    \end{split}
  \end{gather*}
\end{theorem}

Let $H^{p,q}=H^{\otimes p}\otimes (H^*)^{\otimes q}$.
Since $H^{p,q}$ is decomposed into a direct sum of some copies of $V_{\ulla}$ for bipartitions $\ulla$, we obtain the first cohomology of $\Aut(F_n)$ with coefficients in $H^{p,q}$.

\begin{theorem}[Theorem \ref{H1HpHq}]
Let $p,q\ge 1$.

(1)
If $p-q\neq 1$,
then we have
\begin{gather*}
    H^1(\Aut(F_n),H^{p, q})=0\quad \text{for $n\ge 3$}.
\end{gather*}

(2)
If $p-q=1$, then we have
\begin{gather*}
    H^1(\Aut(F_n),H^{q+1,q})=\Q^{\frac{1}{2}(q+2)!}\quad \text{for } n\ge 4.
\end{gather*}
\end{theorem}

Kawazumi and Vespa \cite{Kawazumi-Vespa} proposed a conjectural structure of the stable cohomology of $\Aut(F_n)$ with coefficients in $H^{p,q}$ for $p,q\ge 0$.
The above theorem gives a proof of their conjecture in cohomological degree $1$. Lindell \cite{Lindell22} has recently proved their conjecture in all degrees (see Theorem \ref{introconjKV}) after the first version of the present paper appeared.

\subsection{Conjectural structure of the stable rational cohomology of $\IA_n$}

In this subsection, we make a conjecture about a complete algebraic structure of the stable rational cohomology of $\IA_n$.
To the best of our knowledge, no such conjecture has been given in the literature. 

Cohen and Pakianathan \cite{CP}, Farb \cite{Farb} and Kawazumi \cite{Kawazumi} independently determined the first cohomology of $\IA_n$.
However, the rational cohomology of $\IA_n$ of degree greater than $1$ has not been determined.

The \emph{Albanese cohomology} $H_A^*(\IA_n,\Q)$ of $\IA_n$ is defined as the image of the rational cohomology of the abelianization of $\IA_n$ under the induced map by the abelianization map.
This terminology was introduced in Church, Ellenberg and Farb \cite{Church-Ellenberg-Farb}, but the notion had been studied earlier.
Pettet \cite{Pettet} and the second-named author \cite{KatadaIA} determined the stable $\GLnQ$-module structure of $H_A^2(\IA_n,\Q)$ and $H_A^3(\IA_n,\Q)$, respectively.

Let $W_*$ denote the \emph{traceless part} of the graded-symmetric algebra of $U_*=\bigoplus_{i\ge 1} U_i$, where $U_i=\Hom(H,\bigwedge^{i+1}H)$.
Then the second-named author \cite{KatadaIA} made the following conjecture.

\begin{conjecture}[\cite{KatadaIA}, see Conjecture \ref{conjalb}]\label{introconjalb}
We stably have an isomorphism of graded $\GL(n,\Q)$-representations
$$
  H_A^*(\IA_n,\Q)\cong (W_*)^*,
$$
where $(W_*)^*$ is the graded-dual of $W_*$.
\end{conjecture}

Let $H^*(\IA_n,\Q)^{\GL(n,\Z)}$ denote the $\GL(n,\Z)$-invariant part of $H^*(\IA_n,\Q)$.
Then we have a graded algebra homomorphism induced by the cup product
$$\omega_n: H_A^*(\IA_n,\Q)\otimes H^*(\IA_n,\Q)^{\GL(n,\Z)}\to H^*(\IA_n,\Q).$$

\begin{conjecture}[Conjecture \ref{conjIAn}]\label{introconjIAn}
The morphism $\omega_n$ is a stable isomorphism of graded algebras with $\GL(n,\Z)$-actions.
\end{conjecture}

Day and Putman \cite{Day-Putman} showed that the $\GLnZ$-invariant part of $H^2(\IA_n,\Z)$ vanishes.
We also propose a conjectural structure of the $\GL(n,\Z)$-invariant part of the stable rational cohomology of $\IA_n$.

\begin{conjecture}[Conjecture \ref{conjectureHIAinv}]\label{introconjectureHIAinv}
The $\GLnZ$-invariant part of the rational cohomology of $\IA_n$ stabilizes, and we stably have an isomorphism of graded algebras
\begin{gather*}
    H^*(\IA_n,\Q)^{\GLnZ}\cong \Q[z_1,z_2,\dots], \quad \deg z_i=4i.
\end{gather*}
\end{conjecture}

By Conjectures \ref{introconjalb}, \ref{introconjIAn} and \ref{introconjectureHIAinv}, we obtain the following conjecture about a complete algebraic structure of $H^*(\IA_n,\Q)$.

\begin{conjecture}[Conjecture \ref{conjIAnW}]\label{introconjIAW}
We stably have an isomorphism of graded algebras with $\GLnZ$-actions
\begin{gather*}
    H^*(\IA_n,\Q) \cong (W_*)^* \otimes 
    \Q[z_1,z_2,\dots],\quad \deg z_i=4i.
\end{gather*}
\end{conjecture}

\subsection{Relation to Church and Farb's conjecture}

Kawazumi and Vespa \cite{Kawazumi-Vespa} proposed a conjectural structure of the stable cohomology $H^*(\Aut(F_n),H^{p,q})$, which was previously known to hold for $p=0$ or $q=0$ by \cite{Djament-Vespa, Djament, Vespa, Randal-Williams}.
After the first version of the present paper appeared on arXiv, Lindell \cite{Lindell22} proved the conjecture of Kawazumi and Vespa.

\begin{theorem}[Lindell \cite{Lindell22}, see Theorem \ref{conjKV}]\label{introconjKV}
Let $p,q\ge 0$.
If $i\neq p-q$, then we have
\begin{gather*}
    H^i(\Aut(F_n),H^{p,q})=0 
\end{gather*}
for sufficiently large $n$.
Otherwise, we have
\begin{gather*}
    H^{i}(\Aut(F_n),H^{q+i,q})=\calC_{\calP_0^{\circlearrowright}}(q+i,q)
\end{gather*}
for sufficiently large $n$, where $\calC_{\calP_0^{\circlearrowright}}$ is a certain wheeled PROP (see Remark \ref{remarkKV}).
\end{theorem}

We consider the following stability of the family $\{H^i(\IA_n,\Q)\}_n$ of $\GL(n,\Z)$-representations.
\begin{itemize}
    \item[$(\SF_i)$: ] For some $r\ge 0$, $\{H^i(\IA_n,\Q)\}_{n\ge r}$ is a stable family of $\GL(n,\Z)$-representations, that is, 
    there are finitely many bipartitions $\ul\lambda_j$ (possibly with repetitions) such that for all $n\ge r$ we have an isomorphism of $\GLnZ$-representations
\begin{gather*}
H^i(\IA_n,\Q) \cong \bigoplus_{j} V_{\ul\la_j}(n).
\end{gather*}
\end{itemize}

\begin{conjecture}[Church--Farb \cite{Church-Farb}, see Conjecture \ref{conjPi}]\label{introconjPi}
For each $i\ge0$, the hypothesis $(\SF_i)$ holds.
\end{conjecture}

By using Theorem \ref{introconjKV}, we have the following relation between the above conjectures.

\begin{theorem}[see Theorems \ref{thmHIAinv} and \ref{conjIAprop}]
\label{conjIApropintro}
We have the following.
\begin{itemize}
    \item Conjecture \ref{introconjPi} implies Conjecture \ref{introconjectureHIAinv}.
    \item Conjectures \ref{introconjalb} and \ref{introconjPi} imply Conjecture \ref{introconjIAn}.
\end{itemize} 

Therefore, Conjectures \ref{introconjalb} and \ref{introconjPi} imply Conjecture \ref{introconjIAW}.
\end{theorem}

\subsection{Conjectural structure of the stable rational cohomology of $\IO_n$}

Similarly to the case of $\IA_n$, we make the following conjecture about a complete algebraic structure of the stable rational cohomology of $\IO_n$, which is an analogue of Conjecture \ref{introconjIAW}.

\begin{conjecture}[Conjecture \ref{conjIOnWO}]
We stably have an isomorphism of graded algebras with $\GL(n,\Z)$-actions
$$H^*(\IO_n,\Q)\cong (W^O_*)^*\otimes 
    \Q[z_1,z_2,\dots],\quad \deg z_i=4i,
$$
where $W^O_*$ is the traceless part of the graded-symmetric algebra of $U^O_*=\bigoplus_{i\ge 1}U^O_i$, where $U^O_1=U_1/H$ and $U^O_i=U_i$ for $i\ge 2$.
\end{conjecture}

A conjectural relation between the stable rational cohomology of $\IA_n$ and $\IO_n$ is the following.

\begin{conjecture}[Conjecture \ref{conjstrIOIA}]
We have a stable isomorphism 
$$\psi: H^*(\IO_n,\Q)\otimes H^*(F_n,\Q)\to H^*(\IA_n,\Q)$$
of graded algebraic $\GL(n,\Z)$-representations.
\end{conjecture}

\subsection{Stable rational cohomology of the Torelli groups of surfaces}

Let $\calI_{g}$ (resp. $\calI_{g,1}$) denote the \emph{Torelli group} of a compact oriented surface of genus $g$ (resp. with one boundary component).
Then $\Sp(2g,\Z)$ acts on the rational cohomology of the Torelli groups.

We have injective graded algebra maps
$$\iota_{g}: H_A^*(\calI_{g},\Q)\hookrightarrow H^*(\calI_{g},\Q), \quad
\iota_{g,1}: H_A^*(\calI_{g,1},\Q)\hookrightarrow H^*(\calI_{g,1},\Q).$$
Then the following conjecture determines an algebraic structure of the stable cohomology of the Torelli groups.

\begin{conjecture}[Conjecture \ref{conjIgstr}]
The graded algebra maps $\iota_{g}$ and $\iota_{g,1}$ are stable isomorphisms.
\end{conjecture}

Let $i_{g,1}:\calI_{g,1}\hookrightarrow\IA_{2g}$ denote the Dehn--Nielsen embedding.
We make the following conjecture about the invariant parts of the stable rational cohomology of $\IA_n$ and $\calI_{g,1}$.

\begin{conjecture}[Conjecture \ref{conjIandIAinv}]
The graded algebra homomorphism
\begin{gather*}
i_{g,1}^*: H^*(\IA_{2g},\Q)^{\GL(2g,\Z)} \to H^*(\calI_{g,1},\Q)^{\Sp(2g,\Z)}
\end{gather*}
induced by $i_{g,1}$ is a stable isomorphism.
\end{conjecture}

\begin{remark}
After the first version of the present paper was circulated, Oscar Randal-Williams informed us of the result of Li and Sun about the improvement of the Borel theorem \cite{Li-Sun}.
We have removed our results about Borel's stable range in the first version since they are weaker than Li and Sun's result.
Our results on Borel's stable range has been separated into another paper \cite{HK2} since they give an approach different from Li and Sun's.
Also, we have improved the stable ranges of the low degree cohomology of $\Aut(F_n)$ and $\Out(F_n)$ by using Li and Sun's result.
Moreover, we realized that the assumption in Theorem \ref{conjIApropintro}, i.e. Conjecture \ref{introconjPi}, can be replaced with a weaker assumption that
the family $\{H_i(\IA_n;\Q)\}_n$ is stably algebraic in the sense of Remark \ref{stablyalgebraicBorel}.
\end{remark}

\begin{remark}
    After the first version of the present paper was circulated, Erik Lindell informed us of his proof of the conjecture of Kawazumi--Vespa, and later his preprint appeared on arXiv \cite{Lindell22}.
    In Section \ref{secstrIA},
    we have decided to rely on Lindell's result, which was a conjecture in the previous version, so that our results about the relations between various conjectures become simpler statements.
\end{remark}

\subsection{Organization of the paper}

The rest of this paper is organized as follows.
In Section \ref{alg-glnz-rep}, we recall some facts about representation theory of $\GL(n,\Q)$.
In Section \ref{homologyofGL}, we recall Borel's stability and vanishing theorem for the cohomology of $\GL(n,\Z)$ with coefficients in irreducible algebraic representations $V_{\ulla}$ corresponding to bipartitions.
In Section \ref{secLowDegree}, we compute the low-degree cohomology of $\Aut(F_n)$ with coefficients in $V_{\ul\lambda}$ by using the Hochschild--Serre spectral sequence.
In Section \ref{secNonsemisimple}, we compute the first cohomology of $\Aut(F_n)$ with coefficients in the $\Aut(F_n)$-module of Jacobi diagrams, which is not semisimple.
In Section \ref{secinvIA}, we propose a conjectural structure of the $\GL(n,\Z)$-invariant part of the stable cohomology $H^*(\IA_n,\Q)$ and prove it under the assumption $(\SF_i)$ for each $i$.
In Section \ref{secstrIA}, we propose a conjectural structure of the stable cohomology $H^*(\IA_n,\Q)$ and study relations to some other known conjectures.
In Section \ref{secstrIO}, we study the stable cohomology $H^*(\IO_n,\Q)$ as in the case of $H^*(\IA_n,\Q)$. 
In Section \ref{secConjectures}, we make conjectures about the stable rational cohomology of the Torelli groups of surfaces.

\subsection*{Acknowledgements}
The authors thank Nariya Kawazumi, Erik Lindell, Geoffrey Powell, Oscar Randal-Williams, Takuya Sakasai and Christine Vespa for helpful comments.
K.H. was supported in part by JSPS KAKENHI Grant Number 18H01119 and 22K03311.
M.K. was supported in part by JSPS KAKENHI Grant Number JP22J14812.

\section{Algebraic $\GL(n,\Z)$-representations}
\label{alg-glnz-rep}

By a \emph{polynomial $\GL(n,\Q)$-representation}, we mean a finite-dimensional $\Q[\GL(n,\Q)]$-module $V$ such that after choosing a basis for $V$, the $(\dim V)^2$ coordinate functions are polynomial functions of the $n^2$ variables.
Similarly, a $\GL(n,\Q)$-representation is \emph{algebraic} if the coordinate functions are rational functions.
We refer the reader to \cite{Fulton-Harris} for some standard facts from representation theory.

It is well known that irreducible polynomial $\GL(n,\Q)$-representations are classified by partitions with at most $n$ parts.
A \emph{partition} $\lambda=(\lambda_1,\lambda_2,\dots,\lambda_l)$ is a sequence of non-negative integers such that $\lambda_1\ge\lambda_2\ge\dots\ge\lambda_l$.
Let $l(\lambda)=\max(\{0\}\cup\{i\mid \lambda_{i}> 0\})$ denote the \emph{length} of $\lambda$ and $|\lambda|=\lambda_1+\cdots+\lambda_{l(\lambda)}$ the \emph{size} of $\lambda$.

Let $H=H(n)=\Q^n$ denote the standard representation of $\GL(n,\Q)$. 
We usually omit $(n)$ in the following.
For a partition $\lambda$, let $S^\lambda$ denote
the Specht module for $\lambda$, which is an irreducible representation of $\gpS_{|\lambda|}$.
Define a $\GL(n,\Q)$-representation $$V_\lambda=V_\lambda(n)= H^{\otimes |\lambda|}\otimes_{\Q[\gpS_{|\lambda|}]}S^\lambda.$$
If $l(\lambda)\le n$, then $V_\lambda$ is an irreducible polynomial $\GL(n,\Q)$-representation.
If $l(\lambda)>n$, then we have $V_{\lambda}=0$.

Let $p$ and $q$ be non-negative integers.
Let $H^{p,q}=H^{\otimes p}\otimes (H^*)^{\otimes q}$.
For a pair $(i,j)\in \{1,\dots,p\}\times \{1,\dots,q\}$, the
\emph{contraction map} 
\begin{gather*}
c_{i,j}:H^{p,q}\rightarrow H^{p-1,q-1}
\end{gather*}
is defined by
\begin{gather*}
\begin{split}
&c_{i,j}((v_1\otimes\cdots\otimes v_{p}) \otimes(f_1\otimes\cdots\otimes f_{q}))\\
&\quad=  \langle v_i,f_j\rangle
    (v_1\otimes\cdots\widehat{v_{i}}\cdots\otimes v_{p})\otimes (f_1\otimes\cdots\widehat{f_{j}}\cdots\otimes f_{q})
\end{split}
\end{gather*}
 for $v_1,\ldots, v_p\in H$ and $f_1,\ldots, f_q\in H^*$,
where $\widehat{v_{i}}$ (resp. $\widehat{f_{j}}$) denotes the omission of $v_{i}$ (resp. $f_{j}$), and where $\langle -,-\rangle:H\otimes H^{*}\rightarrow \Q$ denotes the dual pairing.

Let $H^{\langle p,q\rangle}$ denote the \emph{traceless part} of $H^{p,q}$, which is defined by
$$
  H^{\langle p,q\rangle}=\bigcap_{(i,j)\in \{1,\dots,p\}\times \{1,\dots,q\}} \ker c_{i,j}\subset H^{p,q}.
$$

By a \emph{bipartition}, we mean a pair $\ul{\lambda}=(\lambda,\lambda')$ of two partitions $\lambda$ and $\lambda'$.
The \emph{length} of $\ul{\lambda}$ is $l(\ul\lambda)=l(\lambda)+l(\lambda')$.
We define the \emph{degree} of $\ul\lambda$ by
$\deg\ul\lambda=|\lambda|-|\lambda'|\in\Z$, and the \emph{size} of $\ul\lambda$ by $|\ul\lambda|=|\lambda|+|\lambda'|$.
The \emph{dual} of $\ul\lambda$ is defined by $\ul\lambda^*=(\lambda',\lambda)$.

Let $\ul\lambda$ be a bipartition, and set $p=|\lambda|$ and $q=|\lambda'|$.
Let
\begin{gather*}
V_{\ul\lambda}=V_{\ul\lambda}(n)=H^{\langle p,q\rangle}\otimes_{\Q[\gpS_p\times \gpS_q]}(S^{\lambda}\otimes S^{\lambda'}).
\end{gather*}
If $l(\ul\lambda)\le n$, then $V_{\ul\lambda}$ is an irreducible algebraic $\GLnQ$-representation. If $l(\ul\lambda)> n$, then we have $V_{\ul\lambda}=0$.
We have the following decomposition of $H^{\langle p,q\rangle}$ as a $\Q[\GL(n,\Q)\times (\gpS_p\times \gpS_q)]$-module
\begin{gather}\label{kercontraction}
    H^{\langle p,q\rangle}=\bigoplus_{\substack{\ul\lambda=(\lambda,\lambda'): \text{bipartition with} \\ l(\ul\lambda)\le n,\;|\lambda|=p,\;|\lambda'|=q}} V_{\ul\lambda}\otimes (S^{\lambda}\otimes S^{\lambda'}).
\end{gather}
(See \cite[Theorem 1.1]{Koike}.)
It is well known that irreducible algebraic $\GL(n,\Q)$-representations are classified by bipartitions with at most $n$ parts.

For $0\le l\le \min(p,q)$, let
$$\Lambda_{p,q}(l)=\{((i_1,j_1),\ldots, (i_{l+1},j_{l+1}))\in ([p]\times [q])^{l+1}\mid \substack{1\le i_1<i_2<\cdots<i_{l+1}\le p,\\ j_1,  j_2, \ldots, j_{l+1}: \text{ distinct}}\}.$$
For $I=((i_1,j_1),\ldots, (i_{l+1},j_{l+1}))\in \Lambda_{p,q}(l)$, let 
$$c_{I}: H^{p,q}\to  H^{p-l-1,q-l-1}$$
be the composition of contraction maps defined by
\begin{gather*}
    \begin{split}
&(v_1\otimes\cdots\otimes v_{p}) \otimes(f_1\otimes\cdots\otimes f_{q})\\
&\mapsto \left(\prod_{r=1}^{l+1}\langle v_{i_r},f_{j_r}\rangle\right)
(v_1\otimes\cdots\widehat{v_{i_1}}\cdots\widehat{v_{i_{l+1}}}\cdots\otimes v_{p})\otimes (f_1\otimes\cdots\widehat{f_{j_1}}\cdots\widehat{f_{j_{l+1}}}\cdots\otimes f_{q}).
    \end{split}
\end{gather*}
Let $k=\min(p,q)$. Define an increasing filtration $F^*=\{F^l\}_{0\le l\le k}$ 
$$H^{\langle p,q\rangle}=F^0\subset F^1\subset \cdots \subset F^l\subset F^{l+1} \subset \cdots \subset F^k=H^{p,q}$$
of $H^{p,q}$ by
\begin{gather*}
     F^l=\ker \left(\bigoplus_{I\in \Lambda_{p,q}(l)} c_{I}\;:\; H^{p,q}\to \bigoplus_{I\in \Lambda_{p,q}(l)} H^{p-l-1,q-l-1}\right).
\end{gather*}
If $n> p+q-1$, then we have the following exact sequence
\begin{gather}\label{exactfilt}
    0\to F^{l-1}\to F^{l} \to (H^{\langle p-l,q-l\rangle})^{\oplus \binom{p}{l}\binom{q}{l} l!}\to 0
\end{gather} 
for $1\le l\le k$.

Let $\det$ denote the ($1$-dimensional) determinant representation.
Then we have $\det \cong \bigwedge^n V \cong V_{(1^n)}$, where $(1^n)=(1,\dots,1)$ consists of $n$ copies of $1$.
For any bipartition $\ul\lambda=(\la,\la')$ with $l(\ul\la)\le n$, we have 
$$V_{\ul\lambda}\cong V_{\mu}\otimes {\det}^k,$$
where $\mu$ is a partition with at most $n$ parts and $k$ is an integer satisfying
$$
(\lambda_1,\ldots,\lambda_{l(\lambda)},0,\ldots,0,-\lambda'_{l(\lambda')},\ldots,-\lambda'_1)=(\mu_1+k,\ldots,\mu_n+k).
$$

By an \emph{algebraic $\GL(n,\Z)$-representation}, we mean the restriction to $\GLnZ$ of an algebraic $\GL(n,\Q)$-representation.
Note that the square of the determinant representation, ${\det}^2$, is trivial as a $\GL(n,\Z)$-representation.
Therefore, any irreducible algebraic $\GL(n,\Z)$-representation is equivalent to the restriction of an irreducible polynomial $\GL(n,\Q)$-representation to $\GL(n,\Z)$.

\section{Borel's stability and vanishing theorem for cohomology of $\GL(n,\Z)$}\label{homologyofGL}

In \cite{Borel1,Borel2}, Borel computed the cohomology $H^p(\Gamma,V)$ of an arithmetic group $\Gamma$ with coefficients in an algebraic $\Gamma$-representation $V$ in a stable range.
Recently, Li and Sun \cite{Li-Sun} gave an improved Borel's stable range, which is independent of the type of the representation $V$.

\subsection{The stable cohomology of $\GL(n,\Z)$}

Here we recall known results about the stable cohomology of $\GL(n,\Z)$.

\begin{theorem}[Borel \cite{Borel1, Borel2}, Li--Sun \cite{Li-Sun}]\label{Borelbipartition}
(1) For each integer $n\ge 1$, the algebra map $$H^*(\GL(n+1,\Z),\Q)\to H^*(\GL(n,\Z),\Q)$$ induced by the inclusion $\GL(n,\Z)\hookrightarrow \GL(n+1,\Z)$ is an isomorphism in $*\le n-2$.
Moreover, we have an algebra isomorphism
\begin{gather*}
    \varprojlim_n H^*(\GL(n,\Z),\Q)\cong \bigwedge_{\Q} (x_1, x_2,\ldots), \quad \deg x_i=4i+1
\end{gather*}
in degree $*\le n-2$.

(2)
 Let $V$ be an algebraic $\GL(n,\Q)$-representation such that $V^{\GL(n,\Q)}=0$.
 Then we have 
 \begin{gather*}
     H^p(\GL(n,\Z),V) = 0 \quad \text{for $p\le n-2$}.
 \end{gather*}
\end{theorem}

\begin{remark}
Borel's original result is about the real cohomology $H^*(\SLnZ,\R)$.
We stably have $H^*(\SLnZ,\R)=\bigwedge_\R(b_1,b_2,\dots)$, where $b_k\in H^{4k+1}(\SLnZ,\R)$
is the Borel class.
Since the natural map between the cohomology of $\SL(n,\Z)$ and $\GL(n,\Z)$ is stably isomorphism, we consider the case of $\GL(n,\Z)$ in Theorem \ref{Borelbipartition}.
The element $x_k\in H^{4k+1}(\GLnZ,\Q)$ is a scalar multiple $c_kb_k$, where $c_k\in\R\setminus\{0\}$. The choice of $c_k$ does not matter in what follows. 
\end{remark}

Kupers--Miller--Patzt \cite{Kupers-Miller-Patzt} improved the stable range in Theorem \ref{Borelbipartition} in the case of the rational cohomology of $\GL(n,\Z)$.
The following theorem gives the best stable range of the cohomology of $\GL(n,\Z)$ that we know.

\begin{theorem}[Borel \cite{Borel1, Borel2}, Li--Sun \cite{Li-Sun}, Kupers--Miller--Patzt \cite{Kupers-Miller-Patzt}]\label{theoremhomologyofGL}
  Let $p\ge0$. Then we have the following.
\begin{enumerate}
\item  For $n\ge p+1$, we have
\begin{gather*}
     H^*(\GL(n,\Z), \Q)\cong \bigwedge_{\Q} (x_1, x_2,\ldots), \quad \deg x_i=4i+1
\end{gather*}
in cohomological degree $*\le p$.
\item  Let $\ul\lambda\neq (0,0)$ be a bipartition. For $n\ge p+2$, we have
  \begin{gather*}
     H^p(\GL(n,\Z), V_{\ul\lambda})=0.
  \end{gather*}
\end{enumerate}
\end{theorem}

\subsection{Stable families of algebraic $\GL(n,\Z)$-representations}

Here we consider families of $\GL(n,\Z)$-representations indexed by $n\in\N$ that are stable with respect to the types of representations.

\begin{definition}
\label{stablefamily}
Let $0\le s\le r$ be integers.
A family $W_*=\{W_n\}_{n\ge s}$ of $\GLnZ$-representations $W_n$ is called \emph{stable in ranks $\ge r$} if there are finitely many bipartitions $\ul\lambda_i$ (possibly with repetitions) such that for all $n\ge r$ we have a $\GLnZ$-module isomorphism
\begin{gather*}
W_n \cong \bigoplus_{i} V_{\ul\la_i}(n).
\end{gather*}
\end{definition}

A stable family of $\GLnZ$-representations defined above is almost the same as a ``uniformly multiplicity-stable sequence'' in the sense of Church and Farb \cite{Church-Farb}.  (Here we consider only families of modules that are finite dimensional in the stable range.)

Let $W_*$ be a stable family of $\GL(n,\Z)$-representations in ranks $\ge r$ such that we have $W_n\cong \bigoplus_{i} V_{\ul\la_i}(n)$ for $n\ge r$.
For each $n\ge s$ and $p\ge 0$, we have the cup product map
\begin{gather*}
\theta_{n,p}:
H^p(\GL(n,\Z),\Q)\otimes W_n^{\GL(n,\Z)} \to H^p(\GL(n,\Z),W_n).
\end{gather*}
By the work of Borel \cite{Borel2} and Li--Sun \cite{Li-Sun},
it is known that the cup product map $\theta_{n,p}$ is an isomorphism in a stable range.
We will improve the stable range.
For a non-negative integer $p$, set
$$
 n_0(W_*,p)=
 \begin{cases}
     \max(r, p+1)& (W_*=\bigoplus_{\text{finite}}\Q)\\
     \max(r, p+2)& (\text{otherwise}).
 \end{cases}
$$
By Theorem \ref{theoremhomologyofGL}, we have the following proposition.

\begin{proposition}\label{stablefamilyBorel}
  Let $W_*$ be a stable family of $\GL(n,\Z)$-representations.
  Then $\theta_{n,p}$ is an isomorphism for $n\ge n_0(W_*, p)$.
\end{proposition}

\begin{remark}\label{stablyalgebraicBorel}
Using Theorem \ref{theoremhomologyofGL}, we can prove that Proposition \ref{stablefamilyBorel} also holds for \emph{stably algebraic} families, which are weaker variants of stable families.
Here a family $W_*=\{W_n\}_{n\ge s}$ is stably algebraic in ranks $\ge r$ for some $r\ge s$ if $W_n$ is an algebraic $\GLnZ$-representation for all $n\ge r$. 
In this case, the sources and targets of the stable isomorphisms $\theta_{n,p}$ do not stabilize since they depend on $W_n^{\GL(n,\Z)}$.
\end{remark}

\section{The low-degree twisted cohomology of $\Aut(F_n)$ and $\Out(F_n)$}
\label{secLowDegree}

In this section, we compute the low-degree twisted cohomology of $\Aut(F_n)$ and $\Out(F_n)$ by using Theorem \ref{theoremhomologyofGL} and the Hochschild--Serre spectral sequence (see \cite{Brown} for example).
Let $H_{\Z}=H_1(F_n,\Z)$, $H_{\Z}^*=H^1(F_n,\Z)$, $H=H_1(F_n,\Q)\cong\Q^n$ and $H^*=H^1(F_n,\Q)$.

\subsection{Known results about the first (co)homology of $\Aut(F_n)$ and $\Out(F_n)$}
\label{ssec:known results}

Here we list what is known about the first (co)homology of $\Aut(F_n)$ and $\Out(F_n)$.
It is well known that $H_1(\Aut(F_n),\Q)=H_1(\Out(F_n),\Q)=0$ for all $n\ge 1$.

Satoh \cite{Satoh2013, Satoh2006} obtained 
\begin{gather*}
 \begin{split}
    H_1(\Aut(F_n),H_{\Z})&=0, \quad H_1(\Out(F_n),H_{\Z})=0 \quad \text{for}\; n\ge 4,\\
    H_1(\Aut(F_n),H_{\Z}^*)&=\Z, \quad H_1(\Out(F_n),H_{\Z}^*)=\Z/(n-1)\Z\quad \text{for}\; n\ge 4,\\
    H^1(\Aut(F_n),\IA_n^{\ab})&=\Z^2, \quad H^1(\Out(F_n),\IA_n^{\ab})=\Z \quad \text{for}\; n\ge 5,
 \end{split}
\end{gather*} 
and for coefficients $H$ and $H^*$, the results hold for $n\ge 2$ by tensoring $\Q$.

Randal-Williams \cite[Theorem A]{Randal-Williams} obtained
\begin{gather*}
 \begin{split}
     H^1(\Aut(F_n), (H_{\Z}^*)^{\otimes q})&=0, \quad  H^1(\Out(F_n), (H^*)^{\otimes q})=0 \quad \text{for}\; n\ge q+5,\\
    H^1(\Aut(F_n), H^{\otimes q})&=0, \quad  H^1(\Out(F_n), H^{\otimes q})=0 \quad \text{for}\; n\ge q+5,\; q\ne 1.
 \end{split}
\end{gather*}
He obtained similar results for higher degree cohomology as well (see Section \ref{secstrIA1}).

\subsection{First cohomology of $\Aut(F_n)$ and $\Out(F_n)$ with algebraic coefficients}
In this section, we compute the low-degree cohomology of $\Aut(F_n)$ and $\Out(F_n)$ with coefficients induced by irreducible algebraic $\GL(n,\Q)$-representations.

Let $\ul\lambda$ be a bipartition.
For the exact sequence 
$$1\rightarrow \IA_n\rightarrow \Aut(F_n)\rightarrow \GL(n,\Z)\rightarrow 1,$$
we consider the Hochschild--Serre spectral sequence
\begin{gather*}
    E_2^{p,q}=H^p(\GL(n,\Z),H^q(\IA_n,V_{\ul\lambda}))\Rightarrow H^{p+q}(\Aut(F_n),V_{\ul\lambda}).
\end{gather*}
Since $\IA_n$ acts trivially on $V_{\ul\lambda}$,
we have 
$$
 E_2^{0,q}=H^0(\GL(n,\Z),H^q(\IA_n,V_{\ul\lambda}))=(H^q(\IA_n,\Q)\otimes V_{\ul\lambda})^{\GL(n,\Z)}.
$$

By Theorem \ref{theoremhomologyofGL}, if $n\ge p+2$, then we have 
$$
  E_2^{p,0}=H^p(\GL(n,\Z),H^0(\IA_n,V_{\ul\lambda}))=H^p(\GL(n,\Z),V_{\ul\lambda})=0
$$
for $1\le p\le 4$.
By Farb \cite{Farb}, Cohen--Pakianathan \cite{CP} and Kawazumi \cite{Kawazumi}, we have
\begin{gather}
\label{H1IAn}
H^1(\IA_n,\Q)=V_{0,1}\oplus V_{1,1^2}\quad \text{for $n\ge 3$.}
\end{gather}

By the above argument, we obtain the following theorem.

\begin{theorem}\label{H1Aut}
For any bipartition $\ul\lambda\neq (1,0),(1^2,1)$, we have
$$  
 H^1(\Aut(F_n),V_{\ul\lambda})=0 \quad\text{for}\; n\ge 3.  
$$ 
For $\ul\lambda= (1,0),(1^2,1)$, 
  \begin{gather*}
    \begin{split}
        H^1(\Aut(F_n),V_{1,0})&=\Q \quad\text{for}\; n\ge 4,\\
        H^1(\Aut(F_n),V_{1^2,1})&=\Q \quad\text{for}\; n\ge 4.
    \end{split}
  \end{gather*}

\end{theorem}

\begin{proof}
For a bipartition $\ul\lambda\neq (1,0),(1^2,1)$, we have
$$
E_2^{0,1}=(H^1(\IA_n,\Q)\otimes V_{\ul\lambda})^{\GL(n,\Z)}=((V_{0,1}\oplus V_{1,1^2})\otimes V_{\ul\lambda})^{\GL(n,\Z)}=0.
$$
By Theorem \ref{theoremhomologyofGL}, for $n\ge 3$, we have $E_2^{1,0}=0$ and thus we have 
$$H^1(\Aut(F_n),V_{\ul\lambda})=0.$$

For $\ul\lambda=(1,0)$, we have
$$
E_2^{0,1}=(H^1(\IA_n,\Q)\otimes V_{1,0})^{\GL(n,\Z)}=((V_{0,1}\oplus V_{1,1^2})\otimes V_{1,0})^{\GL(n,\Z)}=\Q
$$
for $n\ge 3$.
By Theorem \ref{theoremhomologyofGL}, for $n\ge 4$, we have $E_2^{1,0}=E_2^{2,0}=0$ and thus we have $H^1(\Aut(F_n),V_{1,0})=E_2^{0,1}=\Q.$

For $\ul\lambda=(1^2,1)$, we have
$$
E_2^{0,1}=(H^1(\IA_n,\Q)\otimes V_{1^2,1})^{\GL(n,\Z)}=((V_{0,1}\oplus V_{1,1^2})\otimes V_{1^2,1})^{\GL(n,\Z)}=\Q
$$
for $n\ge3$.
By Theorem \ref{theoremhomologyofGL}, for $n\ge 4$, we have $E_2^{1,0}=E_2^{2,0}=0$ and thus we have $H^1(\Aut(F_n),V_{1^2,1})=E_2^{0,1}=\Q.$
\end{proof}

In a similar way, we have the following theorem for $\Out(F_n)$.

\begin{theorem}\label{H1Out}
For $\ul\lambda\neq (1^2,1)$, we have 
$$H^1(\Out(F_n),V_{\ul\lambda})=0\quad \text{ for } n\ge 3.$$
In particular, we have $H^1(\Out(F_n), V_{1,0})=0$ for $n\ge3$.

For $\ul\lambda=(1^2,1)$, we have
$$H^1(\Out(F_n),V_{1^2,1})=\Q\quad \text{ for } n\ge 4.$$ 
\end{theorem}

\begin{proof}
Let $\IO_n$ be the subgroup of $\Out(F_n)$ appearing in the short exact sequence
\begin{gather}
\label{exactseqIO}
    1\to \IO_n\to \Out(F_n)\to \GLnZ\to 1.
\end{gather}
The proof is the same as that of Theorem \ref{H1Aut} except that we use Kawazumi's computation \cite{Kawazumi}
\begin{gather*}
H^1(\IO(F_n),\Q)=V_{1,1^2}
\end{gather*}
instead of \eqref{H1IAn}
and the Hochschild--Serre spectral sequences for the short exact sequence \eqref{exactseqIO}.
\end{proof}

By Theorems \ref{H1Aut} and \ref{H1Out}, we obtain the following corollary, which partially recovers (and in a few cases improves) results of Satoh \cite{Satoh2013, Satoh2006} and Randal-Williams \cite{Randal-Williams} that we mentioned in Section \ref{ssec:known results}.

\begin{corollary}[Cf. \cite{Satoh2013, Satoh2006,Randal-Williams}]
\begin{gather*}
 \begin{split}
    H_1(\Aut(F_n),H)&=0 ,\quad H_1(\Out(F_n),H)=0\quad \text{for}\; n\ge 3,\\
    H_1(\Aut(F_n),H^*)&=\Q \quad \text{for}\; n\ge 4,\quad 
    H_1(\Out(F_n),H^*)=0 \quad \text{for}\; n\ge 3,\\
    H^1(\Aut(F_n),H_1(\IA_n,\Q))&=\Q^2,\quad H^1(\Out(F_n),H_1(\IA_n,\Q))=\Q \quad \text{for}\; n\ge 4,\\
     H^1(\Aut(F_n), (H^*)^{\otimes q})&=0 ,\quad H^1(\Out(F_n), (H^*)^{\otimes q})=0 \quad \text{for}\; n\ge 3,\\
    H^1(\Aut(F_n), H^{\otimes q})&=0 ,\quad H^1(\Out(F_n), H^{\otimes q})=0 \quad \text{for}\; n\ge 3,\; q\ne 1.
 \end{split}
\end{gather*} 
\end{corollary}

\begin{proof}
The proof is straightforward.
For coefficients in $H^{\otimes q}$ and $(H^*)^{\otimes q}$, we use the decomposition
$H(n)^{\ot q}=\bigoplus_{\lambda}V_\lambda(n)^{\oplus \dim S^{\lambda}}$, where $\lambda$ runs through all partitions with $|\la|=q$ and $l(\la)\le n$.
\end{proof}

The following result complements Randal-Williams's result for twisted cohomology of $\Aut(F_n)$ and $\Out(F_n)$ \cite[Theorem A]{Randal-Williams} in cohomological degree $1$.

\begin{theorem}
\label{H1HpHq}
(1) Let $p,q\ge 1$.
If $p-q\neq 1$,
then we have
\begin{gather*}
    H^1(\Aut(F_n),H^{p, q})=0\quad \text{for $n\ge 3$}.
\end{gather*}
For $\Out(F_n)$, the same result holds.

(2) Let $p,q\ge 1$.
If $p-q=1$, then we have
\begin{gather*}
    H^1(\Aut(F_n),H^{q+1,q})=\Q^{\frac{1}{2}(q+2)!}\quad \text{for } n\ge 4.
\end{gather*}
For $\Out(F_n)$, we have
\begin{gather*}
    H^1(\Out(F_n),H^{q+1,q})=\Q^{\binom{q+1}{2}q!}\quad \text{for } n\ge 4.
\end{gather*}
\end{theorem}

\begin{proof}
(1) By \eqref{kercontraction} and \eqref{exactfilt}, $H^{p, q}$ is a direct sum of some copies of $V_{\ul\lambda}$'s for $\ulla=(\la,\la')$ with $\deg \ul\lambda=p-q$ and $|\ul\lambda|\le p+q$.
Therefore, the statements follow from Theorems \ref{H1Aut} and \ref{H1Out}.

(2)
If $p=q+1$, then we need to count the multiplicities of $V_{1^2,1}$ and $V_{1,0}$ in $H^{q+1,q}$ since we have 
$$H^1(\Aut(F_n),H^{q+1,q})=H^1(\IA_n,H^{q+1,q})^{\GL(n,\Z)}$$
for $n\ge 4$ by the fact that $E_2^{1,0}=E_2^{2,0}=0$ for $n\ge 4$.
By the exact sequence \eqref{exactfilt}, the multiplicity of $V_{1,0}$ in $H^{q+1,q}$ is $\binom{q+1}{q}\binom{q}{q}q!=(q+1)!$, and the multiplicity of $V^{\langle 2,1\rangle}$ in $H^{q+1,q}$ is $\binom{q+1}{q-1}\binom{q}{q-1}(q-1)!=\binom{q+1}{2}q!$.
Therefore, by the decomposition \eqref{kercontraction}, the multiplicity of $V_{1^2,1}$ in $H^{q+1,q}$ is $\binom{q+1}{2}q!$ since we have $\dim (S^{1^2}) \dim (S^{1})=1$.
Therefore, we have
$$H^1(\Aut(F_n),H^{q+1,q})=\Q^{\left((q+1)!+\binom{q+1}{2}q!\right)}=\Q^{\frac{1}{2}(q+2)!}.$$

For $\Out(F_n)$, we only need the multiplicity of $V_{1^2,1}$ in $H^{q+1,q}$, which is $\binom{q+1}{2}q!$ by the above argument.
Therefore, we have 
$$H^1(\Out(F_n),H^{q+1,q})=H^1(\IO_n,H^{q+1,q})^{\GL(n,\Z)}
=\Q^{\binom{q+1}{2}q!}$$
for $n\ge 4$.
\end{proof}

\begin{remark}\label{remarkKV}
The vector spaces $H^1(\Aut(F_n),H^{p,q})$ and $H^1(\Out(F_n),H^{p,q})$ has $\gpS_p\times\gpS_q$-module structure.
Conjecture 6 of \cite{Kawazumi-Vespa} (see Theorem \ref{conjKV}) implies that for sufficiently large $n$, we have
$$H^1(\Aut(F_n),H^{q+1,q})= \calC_{\calP_0^{\circlearrowright}}(q+1,q),$$
where $\calC_{\calP_0^{\circlearrowright}}$ is the wheeled PROP associated to the wheeled completion of the operad $\calP_0$, which is the operadic suspension of the operad $\calC om$ (see also \cite[Conjecture 12.2]{KatadaIA}).
We can check that $\calC_{\calP_0^{\circlearrowright}}(q+1,q)=\Q^{\frac{1}{2}(q+2)!}$.
Therefore, the above theorem gives a proof of the degree $1$ part of their conjecture. Note that Lindell \cite{Lindell22} has recently proved their conjecture in all degrees.
\end{remark}

\subsection{Second cohomology of $\Aut(F_n)$ and $\Out(F_n)$ with algebraic coefficients}\label{secondcohomology}

Here we consider the second cohomology of $\Aut(F_n)$ and $\Out(F_n)$ and apply it to the second cohomology of $\IA_n$ and $\IO_n$.

\begin{theorem}\label{H2Aut}
(1)
For any bipartition $\ul\lambda=(\lambda,\lambda')\neq (0,0)$, we have
$$
H^2(\Aut(F_n),V_{\ul\lambda})\cong (H^2(\IA_n,\Q)\otimes V_{\ul\lambda})^{\GL(n,\Z)} \quad \text{ for }n\ge 5.
$$
For the trivial coefficient, we have 
$$
H^2(\Aut(F_n),\Q)\cong H^2(\IA_n,\Q)^{\GL(n,\Z)} \quad \text{ for }n\ge 4.
$$

(2)
For any bipartition $\ul\lambda=(\lambda,\lambda')\neq (0,0)$, we have
$$
H^2(\Out(F_n),V_{\ul\lambda})\cong (H^2(\IO_n,\Q)\otimes V_{\ul\lambda})^{\GL(n,\Z)} \quad \text{ for }n\ge 5.
$$
For the trivial coefficient, we have 
$$
H^2(\Out(F_n),\Q)\cong H^2(\IO_n,\Q)^{\GL(n,\Z)} \quad \text{ for }n\ge 4.
$$
\end{theorem}

\begin{proof}
(1)
For $n\ge 5$, we have 
$E_2^{2,0}=E_2^{3,0}=0$.
For 
$n\ge 4,$
we have 
$E_2^{1,1}=E_2^{2,1}=0$.
Therefore, we have
$$H^2(\Aut(F_n),V_{\ul\lambda})\cong E_2^{0,2}=(H^2(\IA_n,\Q)\otimes V_{\ul\lambda})^{\GL(n,\Z)}.$$

The case of the trivial coefficient follows from Theorem \ref{theoremhomologyofGL} in a similar way.

(2) The proof is similar to that of (1).
\end{proof}

Day and Putman \cite{Day-Putman} showed that $H_2(\IA_n,\Z)_{\GL(n,\Z)}=0$ for $n\ge 6$.
The rational version of this result follows from Theorem \ref{H2Aut}.

\begin{corollary}\label{DayPutmanrational}
For $n\ge 4$, we have
$$
H^2(\IA_n,\Q)^{\GL(n,\Z)}=H^2(\IO_n,\Q)^{\GL(n,\Z)}=0.
$$
\end{corollary}

\begin{proof}
We have $H^2(\Aut(F_n),\Q)=H^2(\Out(F_n),\Q)=0$ for $n\ge 1$ by Hatcher--Vogtmann \cite{Hatcher-Vogtmann}.
Therefore, by Theorem \ref{H2Aut}, we have 
$$H^2(\IA_n,\Q)^{\GL(n,\Z)}=H^2(\IO_n,\Q)^{\GL(n,\Z)}=0$$
for $n\ge 4$.
\end{proof}

\section{Stable families of polynomial $\Aut(F_n)$-modules}
\label{secNonsemisimple}

In this section, we apply the results of the previous sections to non-semisimple $\Aut(F_n)$-modules.

\subsection{Polynomial $\Aut(F_n)$-modules}

Here we introduce polynomial $\Aut(F_n)$-modules which may relate to other well-known notions of polynomiality.

\begin{definition}
By a \emph{polynomial $\Aut(F_n)$-module}, we mean an $\Aut(F_n)$-module $M$ which has a composition series of finite length 
$$M=M_0\supsetneq M_1\supsetneq M_2\supsetneq \cdots\supsetneq M_k=0$$
such that each composition factor $M_j/M_{j+1}$ is induced by an irreducible algebraic $\GL(n,\Z)$-representation.

A family $M(*)=\{M(n)\}_{n\ge 0}$ of polynomial $\Aut(F_n)$-modules $M(n)$ is \emph{stable} in ranks $\ge r$ if there is a decreasing filtration
$$
M(*)=M_0(*)\supsetneq M_1(*)\supsetneq M_2(*)\supsetneq \cdots\supsetneq M_k(*)=0
$$
such that each $M_j(*)/M_{j+1}(*)$ is a stable family of $\GL(n,\Z)$-representations in ranks $\ge r$.
Let $\gr M(*)=\bigoplus_{j=0}^{k-1} M_j(*)/M_{j+1}(*)$.
We set $$n_0(M(*),p)=n_0(\gr M(*),p).$$
\end{definition}

Polynomial $\Out(F_n)$-modules and stable families of polynomial $\Out(F_n)$-modules can be defined just in the same way as $\Aut(F_n)$.

\begin{proposition}
\label{propVanshing}
 Let $M(*)$ be a stable family of polynomial $\Aut(F_n)$-modules with no composition factors that are isomorphic to either $V_{1,0}$ or $V_{1^2,1}$.
 Then we have 
 \begin{gather}\label{H1AutM}
 H^1(\Aut(F_n),M(n))=0
 \end{gather}
 for $n\ge n_0(M(*),1)$.
\end{proposition}

\begin{proof}
Let $n\ge n_0(M(*),1)$.
We use the induction on the length $k$ of the decreasing filtration of $M(*)$.
If $k=1$, then $M(n)$ is an irreducible algebraic $\GL(n,\Z)$-representation $V_{\ul\lambda}$ such that $\ul\lambda\neq (1,0), (1^2,1)$. 
Therefore, \eqref{H1AutM} follows from Theorem \ref{H1Aut}.
Suppose that \eqref{H1AutM} holds for $k-1$.
For $M(*)$ with length $k$, the exact sequence of $\Aut(F_n)$-modules
$$
0\to M_1(n)\to M(n) \to M(n)/{M_1(n)} \to 0
$$
induces a long exact sequence 
$$
\cdots \to H^1(\Aut(F_n),M_1(n))\to H^1(\Aut(F_n), M(n)) \to H^1(\Aut(F_n),M(n)/{M_1(n)}) \to\cdots.
$$
By the hypothesis of the induction and Theorem \ref{H1Aut}, we have 
$$
H^1(\Aut(F_n), M(n)/{M_1(n)})=H^1(\Aut(F_n),M_1(n))=0
$$
for $n\ge n_0(M(*),1)$, so we have $H^1(\Aut(F_n),M(n))=0$.
\end{proof}

For stable families of polynomial $\Out(F_n)$-modules, we obtain the following.

\begin{proposition}\label{H1OutM}
 Let $M(*)$ be a stable family of polynomial $\Out(F_n)$-modules with no composition factors that are isomorphic to $V_{1^2,1}$.
 Then we have 
 \begin{gather*}
      H^1(\Out(F_n),M(n))=0
 \end{gather*}
 for $n\ge n_0(M(*),1)$.
\end{proposition}

\subsection{Applications to the $\Aut(F_n)$-modules of Jacobi diagrams}

Massuyeau and the first-named author \cite{Habiro-Massuyeau2} extended the Kontsevich integral of links to a functorial invariant defined on the category of \emph{bottom tangles in handlebodies}. This invariant takes values in the category $\mathbb{A}$ of \emph{Jacobi diagrams in handlebodies}, where Jacobi diagrams are certain uni-trivalent graphs.

The second-named author \cite{Katada1, Katada2} studied the $\Aut(F_n)$-module structure on
the vector space $A(n)=\bigoplus_{d\ge0}A_d(n)$ of Jacobi diagrams on $n$ oriented arcs, where the action is induced by the composition of morphisms in the category $\mathbb{A}$. 
For $d=0$, we have $A_0(n)=\Q$. For $d>0$, $A_d(n)$ admits a filtration
\begin{gather*}
    A_d(n)=A_{d,0}(n)\supset A_{d,1}(n)\supset\dots\supset A_{d,2d-1}(n)=0,
\end{gather*}
whose associated graded $A_{d,k}(n)/A_{d,k+1}(n)$ is identified with the vector space $B_{d,k}(n)$ of open Jacobi diagrams  of degree $d$ with univalent vertices colored by elements of the $\GL(n,\Q)$-module $(\Q^n)^*$.
It follows that $\{B_{d,k}(n)\}_n$ is a stable family of $\GLnQ$-representations. Hence $\{A_d(n)\}_n$ is a stable family of polynomial $\Aut(F_n)$-modules.
Note here that the $\Aut(F_n)$-module $A_d(n)$ does not factor through a $\GL(n,\Z)$-module if $d\ge 2$.

We have the following result.

\begin{proposition}
For each $d\ge0$, we have
\begin{gather*}
    H^1(\Aut(F_n), A_d(n))=H^1(\Aut(F_n), A_d(n)^*)=0
\end{gather*}
for $n\ge 3$.
\end{proposition}

\begin{proof}
For $d>0$, since each irreducible component of $B_d(n)$ is of type $V_{0,\lambda}$ with $2\le |\lambda|\le 2d$, it is easily seen that $A_d(n)$ and $A_d(n)^*$ do not include composition factors isomorphic to $V_{0,1}$ or $V_{1,1^2}$.
For example, $A_1(n)$ is isomorphic to $V_{0,2}$, and the $\Aut(F_n)$-module $A_2(n)$ has the following composition series
$$
A_2(n)\supsetneq A''_2(n)\supsetneq A_{2,1}(n)\supsetneq A_{2,2}(n)\supsetneq 0
\quad (n\ge3)
$$
whose composition factors are
$V_{0,4}, V_{0,2^2}, V_{0, 1^3}$ and $V_{0,2}$ (see \cite{Katada1}).
Hence Proposition \ref{propVanshing} implies the desired result.
\end{proof}

\begin{remark}\label{Ad}
Let $d\ge0$ and $i\ge1$.
By \cite{Djament-Vespa} and the fact that $A_d$ is a contravariant polynomial functor on the category $\mathbf{F}$ of finitely generated free groups and homomorphisms \cite{Katada1}, we obtain 
$\varprojlim_n H^i(\Aut(F_n),A_d(n))=0.$
Considering the stable range, one can show that 
\begin{gather*}
    H^i(\Aut(F_n), A_d(n))=0, \quad n \ge 2i+2d+3
\end{gather*}
and if $i\ge 2d+1$, then one can also show that 
$$H^i(\Aut(F_n),A_d(n)^*)=0, \quad n \ge 2i+2d+3$$
by using Theorem \ref{conjKV}
below.
If $i< 2d+1$, then the latter twisted cohomology can be non-trivial. For example, we have $H^3(\Aut(F_n), A_2(n)^*)\neq 0$ for $n\ge 13$. 
\end{remark}

\begin{remark}
Vespa \cite{Vespa22} generalized the functor $A_d$ to another contravariant polynomial functor $A_d(m,-)_0$ of degree $2d$ on $\mathbf{F}$ by considering ``beaded'' Jacobi diagrams on oriented arcs.
If $m>0$, then $\{A_d(m,n)_0\}_{n\ge 0}$ is \emph{not} a stable family of polynomial $\Aut(F_n)$-modules since $A_d(m,n)_0$ is infinite-dimensional for $d,n>0$.
However, 
each composition factor of $A_d(m,n)_0$ is of the form $V_\lambda(n)^*$, $|\lambda|\le2d$.
Similarly to Remark \ref{Ad}, one can show that
\begin{gather*}
    H^i(\Aut(F_n), A_d(m,n)_0)=0, \quad n \ge 2i+2d+3.
\end{gather*}
\end{remark}

The second-named author \cite{Katada1} proved that the $\Aut(F_n)$-module structure on $A_d(n)$ induces an $\Out(F_n)$-module structure on $A_d(n)$.
Therefore, $\{A_d(n)\}_n$ is also a stable family of polynomial $\Out(F_n)$-modules. Then by Proposition \ref{H1OutM}, we obtain the first cohomology of $\Out(F_n)$ in a similar way.

\begin{proposition}
For each $d\ge0$, we have
\begin{gather*}
    H^1(\Out(F_n), A_d(n))=H^1(\Out(F_n), A_d(n)^*)=0
\end{gather*}
for $n\ge 3$.
\end{proposition}

\section{$\GL(n,\Z)$-invariant part of the stable rational cohomology of $\IA_n$}\label{secinvIA}

So far, no non-trivial elements of the $\GL(n,\Z)$-invariant part of the stable rational cohomology $H^*(\IA_n,\Q)$ have been found.
Igusa \cite{Igusa} constructed higher Franz--Reidemeister torsion classes in $H^{4i}(\IO_n,\Q)^{\GL(n,\Z)}$.
Morita, Sakasai and Suzuki \cite{Morita-Sakasai-Suzuki} constructed the ``secondary classes'' in $H^{4i}(\IO_n,\Q)^{\GL(n,\Z)}$ and $H^{4i}(\IA_n,\Q)^{\GL(n,\Z)}$ ($i\ge1$).
Their classes for $\IO_n$ coincide with Igusa's torsion classes.
It is not known whether or not these classes are non-trivial.


In this section, we propose a conjectural structure of the $\GL(n,\Z)$-invariant part of the stable rational cohomology $H^*(\IA_n,\Q)$, and prove it under Church and Farb's stability conjecture (Conjecture \ref{conjPi}).

\subsection{Conjectural stable structure of $H^*(\IA_n,\Q)^{\GLnZ}$}
Let us recall some facts and conjectures about  $H^*(\IA_n,\Q)^{\GLnZ}$.
We have $H^1(\IA_n,\Z)^{\GL(n,\Z)}=0$ by \cite{CP, Farb, Kawazumi}.
As we mentioned in Section \ref{secondcohomology}, we also have 
$H^2(\IA_n,\Z)^{\GL(n,\Z)}=0$ by \cite{Day-Putman}.
Church and Farb \cite[Conjecture 6.5]{Church-Farb} conjectured that the $\GL(n,\Z)$-invariant part of $H^*(\IA_n,\Q)$ vanishes stably.
We make the following conjecture, which contradicts the above conjecture of Church and Farb.

\begin{conjecture}\label{conjectureHIAinv}
The $\GLnZ$-invariant part of the rational cohomology of $\IA_n$ stabilizes, that is, 
for each $i\ge0$, there is an integer $N\ge0$ such that for all $n\ge N$ the inclusion map $\IA_n\hookrightarrow\IA_{n+1}$ induces an isomorphism
\begin{gather*}
H^i(\IA_{n+1},\Q)^{\GL(n+1,\Z)}\congto
H^i(\IA_n,\Q)^{\GL(n,\Z)}.
\end{gather*}
Moreover, we stably have an isomorphism of graded algebras
\begin{gather*}
    H^*(\IA_n,\Q)^{\GLnZ}\cong \Q[z_1,z_2,\dots], \quad \deg z_i=4i.
\end{gather*}
\end{conjecture}

\begin{remark}
    We expect that the generators $z_i$ of our conjectural structure of the $\GL(n,\Z)$-invariant part of $H^{*}(\IA_n,\Q)$ can be constructed by using the Borel classes. See Proposition \ref{IAnGLinv}.
\end{remark}

We will prove Conjecture \ref{conjectureHIAinv} by assuming the following conjecture, which
follows from \cite[Conjectures 6.3, 6.6 and 6.7]{Church-Farb}.

\begin{conjecture}[Church--Farb \cite{Church-Farb}]
\label{conjPi}
For each $i\ge0$, the following hypothesis $(\SF_i)$ holds.
\begin{itemize}
    \item[$(\SF_i)$]$\{H^i(\IA_n,\Q)\}_n$ is a stable family of $\GL(n,\Z)$-representations.
\end{itemize}
\end{conjecture}

It is known that $(\SF_i)$ holds for $i=0,1$.

The proof of the following theorem will be given in Section \ref{proofthmHIAinv}.

\begin{theorem}\label{thmHIAinv}
Conjecture \ref{conjPi} implies Conjecture \ref{conjectureHIAinv}.
\end{theorem}

\begin{remark}
    In Theorem \ref{thmHIAinv}, Conjecture \ref{conjPi} can be replaced with the weaker assumption that the family $\{H^i(\IA_n,\Q)\}_n$ is stably algebraic. See Remark \ref{algverIAnGLinv}.
\end{remark}

\subsection{Antitransgression maps}

Here we discuss ``antitransgression'' maps for the Hochschild--Serre spectral sequences for certain group extensions. 

Let $G$ be a group such that $H^i(G,\Q)=0$ for $i=k,k+1$, where $k\ge1$.
Consider an exact sequence of groups
\begin{gather*}
    1\to N\to G \to Q \to 1
\end{gather*}
and the Hochschild--Serre spectral sequence
\begin{gather*}
    E_2^{p,q}=H^p(Q,H^q(N,\Q))\Rightarrow H^{p+q}(G,\Q).
\end{gather*}

Since $H^k(G,\Q)=H^{k+1}(G,\Q)=0$, we have $E_{k+2}^{0,k}=E_{k+2}^{k+1,0}=0$ and hence
the map $d_{k+1}^{0,k}:E_{k+1}^{0,k}\to E_{k+1}^{k+1,0}$ is an isomorphism.
Define a $\Q$-linear map
$$\varphi_k = \iota_k (d_{k+1}^{0,k})^{-1} \pi_{k+1}
: E_2^{k+1,0}\to E_2^{0,k},$$
where
$\pi_{k+1}:E_2^{k+1,0}\twoheadrightarrow E_{k+1}^{k+1,0}$
and
$\iota_k : E_{k+1}^{0,k}\hookrightarrow E_2^{0,k}$
are the edge maps.
We call the map $\varphi_k$ the $k$th \emph{antitransgression map} since it runs in the opposite direction of a transgression map (if it exists).

\subsection{Graded-commutative multiplicative spectral sequences}

Here we use Zeeman's comparison theorem to obtain certain isomorphism of graded algebras.

Let $V=\bigoplus_{k\ge2}V_k$ be a graded vector space.
Let $V[-1]=\bigoplus_{k\ge1}V[-1]_k$, $V[-1]_k=V_{k+1}$, be the graded vector space obtained from $V$ by degree-shift by $-1$.
Let $SV=\bigoplus_{i\ge 0}S^i(V)$ and $S(V[-1])=\bigoplus_{i\ge 0}S^i(V[-1])$ be the graded-symmetric algebras on $V$ and $V[-1]$, respectively.

A \emph{multiplicative} spectral sequence $E=\{E_r^{p,q}\}$ is a spectral sequence with associative, unital multiplication maps $E_r^{p,q}\otimes E_r^{p',q'}\to E_r^{p+p',q+q'}$.
A multiplicative spectral sequence $E$ is \emph{graded-commutative} if $ab=(-1)^{(p+q)(p'+q')}ba$ for any $a\in E_r^{p,q}$ and $b\in E_r^{p',q'}$.
For example, the cohomological Hochschild--Serre spectral sequences for exact sequences of groups are graded-commutative multiplicative first-quadrant spectral sequences.

We have the following theorem, which is closely related to Borel's result \cite{Borel0} (see \cite[Theorem 3.27]{McCleary}).

\begin{theorem}\label{stableZeeman}
Let $N,P,Q\in\{0,1,2,\dots,\infty\}$.
Let $E$ be a graded-commutative, multiplicative first-quadrant spectral sequence over $\Q$ satisfying the following conditions

(1) the multiplication map $E_2^{p,0}\ot E_2^{0,q}\to E_2^{p,q}$ is a linear isomorphism for $p\le P$, $q\le Q$,

(2) $E_\infty^{p,q}=0$ if $0<p+q\le N$,

(3) there is a graded vector space $V=\bigoplus_{k\ge2}V_k$ such that we have an isomorphism of graded algebras $E_2^{*,0}\cong SV$ for $*\le P$.

Then the restriction $\varphi_k|_{V_{k+1}}$ of the antitransgression map $\varphi_k:E_2^{k+1,0}\to E_2^{0,k}$
is an isomorphism for each $k\ge1$,
and 
we have an isomorphism of graded algebras
\begin{gather*}
    E_2^{0,*}\cong S(V[-1])
\end{gather*}
for $*\le \min(P-2,Q,N-2)$.
Here, $V[-1]_k$ is identified with the image $\varphi_k(V_{k+1})$.
\end{theorem}

\begin{proof}
We first consider the case where at least one of $P$, $Q$ and $N$ are finite.
Let $K=2+\min(P-2,Q,N-2)$.
For $2\le k\le K$, let ${}^kE$ denote the graded-commutative multiplicative first-quadrant spectral sequence given by 
$${}^kE_2^{p,q}=S(V_k)_p\otimes S(V[-1]_{k-1})_q$$ with $d_j=0$ for $j\neq k$ and $d_k^{0,k-1}=\id_{V_k}:{}^kE_k^{0,k-1}\congto{}^kE_k^{k,0}$.
Since the differential is determined by the Leibniz rule, we have an exact sequence
$$
0\to S^i(V[-1]_{k-1})\to S^1(V_k)\otimes S^{i-1}(V[-1]_{k-1})\to \cdots\to S^i(V_k)\to 0
$$
for each $i\ge 1$.
Therefore, we have ${}^kE_{k+1}^{p,q}=\cdots={}^kE_{\infty}^{p,q}=0$ for $p+q>0$.
Set ${}^{\le k}E=\bigotimes_{2\le i\le k}{}^iE$ for $2\le k\le K$.

We claim that 
the restriction of $\varphi_{k-1}:E_2^{k,0}\to E_2^{0,k-1}$ to $V_{k}$ is an isomorphism for $k=2,\dots,K$, and that
there are morphisms of multiplicative spectral sequences
\begin{gather*}
    {}^{k}f : {}^{\le k}E \to E
\end{gather*}
for $k=2,\dots,K$
such that ${}^{k}f_{2}^{p,0}$ is an isomorphism for $0\le p\le k$, and such that ${}^{k+1}f|_{{}^{\le k}E}={}^k f$ for each $k=2,\dots,K-1$.

The proof of this claim is done by induction on $k$.
For $k=2$, the differential $d_2^{0,1}:E_2^{0,1}\to E_2^{0,2}$ is an isomorphism by conditions (2) and (3). Therefore, the inverse map $\varphi_{1}:E_2^{2,0}=V_2\to E_2^{0,1}$ is an isomorphism, and we have such a morphism ${}^{2}f:{}^{\le2}E\to E$.

Let $k\ge2$.
Suppose that such a morphism ${}^{k}f$ exists.
Then since ${}^{k}f_{2}^{p,0}$ is an isomorphism for $0\le p\le k$ and ${}^{k}f_{\infty}^{p,q}$ is an isomorphism for any $0<p+q\le N$ by the condition (2), it follows from Zeeman's comparison theorem \cite{Zeeman} that ${}^{k}f_{2}^{0,q}$ is an isomorphism for $0\le q\le k-2$. 
Therefore, by the condition (1), ${}^{k}f_{2}^{p,q}$ is an isomorphism for $0\le p\le k,\; 0\le q\le k-2$.
In particular, ${}^{k}f_{2}^{p,k-p}$ is an isomorphism for $2\le p\le k-1$.
Since we have $E_2^{1,0}=0$, by condition (1), we have $E_2^{1,k-1}=0$.
Therefore, the restriction of the edge map $\pi_k:E_2^{k+1,0}\to E_{k+1}^{k+1,0}$ to $V_{k+1}$ is an isomorphism.
Then the restriction of $\varphi_{k}=\iota_{k}(d_{k+1}^{0,k})^{-1}\pi_k :E_2^{k+1,0}\to E_2^{0,k}$ to $V_{k+1}$ is an isomorphism.
Hence, we have a morphism ${}^{k+1}f$ extending ${}^kf$ such that ${}^{k+1}f_{2}^{p,0}$ is an isomorphism for $0\le p\le k+1$.

We now have a morphism ${}^{K}f:{}^{\le K}E\to E$. By Zeeman's comparison theorem, ${}^{K}f_2^{0,q}:{}^{\le K}E_2^{0,q}\to E_2^{0,q}$ is an isomorphism for $q\le K-2$, which implies $E_2^{0,*}\cong S(V[-1])$ for $*\le K-2$.

Let us consider the case where $P=Q=N=\infty$.
The above proof for the finite case shows that there is a family of morphisms ${}^kf:{}^{\le k}E\to E$ for $k\ge2$.
Let ${}'E=\varinjlim_k {}^{\le k}E=\bigotimes_{k\ge 2}{}^{k}E$ be the direct limit of multiplicative spectral sequence.
There is a morphism ${}'f:{}'E\to E$ extending ${}^{k}f$ for all $k\ge2$.
By Zeeman's comparison theorem, we obtain $E_2^{0,*}\cong S(V[-1])$ in a similar way.
\end{proof}

\subsection{The classes $y_k\in H^{4k}(\IA_n,\Q)^{\GLnZ}$}\label{antiAut}

We apply the antitransgression map to the exact sequence
\begin{gather*}
    1\to \IA_n\to \Aut(F_n)\to \GLnZ\to 1.
\end{gather*}
By Theorem \ref{theoremhomologyofGL}, we have for $n\ge p+1$
\begin{gather*}
    E_2^{p,0}=\left(\text{the degree $p$ part of }\bigwedge_{\Q}(x_1,x_2,\dots)\right),
\end{gather*}
where $\deg x_k= 4k+1$.

For $k\ge1$, set
\begin{gather*}
    y_k = \varphi_{4k}(x_k) \in E_2^{0,4k}=H^{4k}(\IA_n,\Q)^{\GLnZ}.
\end{gather*}
We expect that $y_k$ is a nonzero scalar multiple of the ``secondary cohomology class'' $T\beta_{2k+1}^0\in H^{4k}(\IA_n,\R)$ appearing in Morita, Sakasai and Suzuki \cite{Morita-Sakasai-Suzuki}.

\subsection{Proof of Theorem \ref{thmHIAinv}}
\label{proofthmHIAinv}

Let $y_k$ be the elements defined in Section \ref{antiAut}.
Theorem \ref{thmHIAinv} follows from the following.

\begin{proposition}\label{IAnGLinv}
Let $P, Q\ge 0$.
Suppose $(\SF_i)$ holds for any $i\le Q$.
Let $n\ge \max\{n_0(\{H^i(\IA_n,\Q)\}_n, P)\}_{i\le Q}$.
Then we have a stable isomorphism of algebras
\begin{gather*}
    H^*(\IA_n,\Q)^{\GLnZ}\cong \Q[y_1,y_2,\dots]
\end{gather*}
in degree $*\le \min(P-2, Q, \frac{n}{2}-3)$.
\end{proposition}

\begin{proof}
Consider the cohomological Hochschild--Serre spectral sequence for the exact sequence
$1\to \IA_n\to \Aut(F_n)\to \GL(n,\Z)\to 1$,
which is graded-commutative and multiplicative.
By the assumption $(\SF_i)$ for $i\le Q$, it follows from Proposition \ref{stablefamilyBorel} that we have
$$E_2^{p,q}=H^p(\GL(n,\Z),H^q(\IA_n,\Q))\cong H^p(\GL(n,\Z),\Q)\otimes (H^q(\IA_n,\Q))^{\GL(n,\Z)}$$ for $0\le p\le P$, $0\le q\le Q$.
By Galatius \cite{Galatius}, we have $H^i(\Aut(F_n),\Q)=0$ for $0<i\le \frac{n-2}{2}$. Therefore, we have $E_{\infty}^{p,q}=0$ for $0< p+q\le \frac{n-2}{2}$.
By Theorem \ref{theoremhomologyofGL}, we have $E_2^{*,0}\cong \bigwedge_{\Q}(x_1,\ldots)$ for $*\le n-1$.
Therefore, by Theorem \ref{stableZeeman}, we have 
an isomorphism of graded algebras 
\begin{gather*}
    H^*(\IA_n,\Q)^{\GLnZ}\cong S(V[-1])
\end{gather*}
for $*\le \min(P-2, Q, \frac{n}{2}-3)$, where $V=\bigoplus_{k\ge1} \Q x_k$.
Since we have $V[-1]=\bigoplus_{k\ge1} \Q y_k$, $S(V[-1])=\Q[y_1,y_2,\ldots]$.
\end{proof}

\begin{remark}\label{algverIAnGLinv}
By Remark \ref{stablyalgebraicBorel}, we can deduce Proposition \ref{IAnGLinv} with the assumption $(\SF_i)$ 
replaced with the following weaker assumption:
\begin{itemize}
    \item[($\mathrm{SA}_i$)]  $\{H^{i}(\IA_n,\Q)\}_n$ is a stably algebraic family. 
\end{itemize}
\end{remark}

\section{Stable rational cohomology of $\IA_n$}\label{secstrIA}

The whole structure of the stable rational cohomology of $\IA_n$ is still mysterious in degree $\ge 2$.
In this section, we propose a conjectural structure of the stable rational cohomology of $\IA_n$, and study some relations to other known conjectures.

\subsection{Known results and conjectures}\label{secstrIA1}

In \cite{KatadaIA}, the second-named author studied the Albanese homology $H^A_*(\IA_n,\Q)$ of $\IA_n$, which is predual to the Albanese cohomology in the sense of Church, Ellenberg and Farb \cite{Church-Ellenberg-Farb}.
The \emph{Albanese homology} $H^A_*(\IA_n,\Q)$ of $\IA_n$ is defined by
$$
  H^A_*(\IA_n,\Q)=\im\left(H_*(\IA_n,\Q)\xrightarrow{\tau_*} H_*(U,\Q)\right),
$$
where $U=H_1(\IA_n,\Q)$ and $\tau_*$ is
induced by the abelianization map $\tau:\IA_n\to U$.
The \emph{Albanese cohomology} $H_A^*(\IA_n,\Q)$ of $\IA_n$ is defined by
$$
  H_A^*(\IA_n,\Q)=\im\left(H^*(U,\Q)\xrightarrow{\tau^*} H^*(\IA_n,\Q)\right),
$$
which is dual to $H^A_*(\IA_n,\Q)$.

Note that we have $H_A^*(\IA_n,\Q)^{\GL(n,\Z)}=0$ since $H_A^*(\IA_n,\Q)$ is a quotient of $H^*(U,\Q)$ and since we have $H^*(U,\Q)^{\GL(n,\Z)}=0$.

Let $U_i=\Hom(H,\bigwedge^{i+1}H)$ for $i\ge 1$.
Let $S^*(U_*)$ denote the graded-symmetric algebra of the graded $\GL(n,\Q)$-representation $U_*=\bigoplus_{i\ge 1} U_i$.
The traceless part $W_*=\bigoplus_{i\ge 0} W_i$ of $S^*(U_*)$ is defined by using contraction maps (see \cite{KatadaIA} for details).
The second-named author \cite{KatadaIA} proved that 
$H^A_i(\IA_n,\Q)$ includes a $\GL(n,\Q)$-subrepresentation which is isomorphic to $W_i$, and made the following conjecture, which holds for $i=1$ \cite{CP,Farb,Kawazumi}, for $i=2$ \cite{Pettet} and for $i=3$ \cite{KatadaIA}.

\begin{conjecture}[\cite{KatadaIA}]\label{conjalb}
Let $i\ge 0$.
We have a $\GL(n,\Q)$-isomorphism
$$
  H^A_i(\IA_n,\Q)\cong W_i
$$
for sufficiently large $n$.
\end{conjecture}

For any algebraic $\GL(n,\Q)$-representation $V$,
the inclusion map $\IA_n\hookrightarrow \Aut(F_n)$ induces a morphism of graded vector spaces
\begin{gather}\label{AutIA}
     i^*:  H^*(\Aut(F_n),V)\to H^*(\IA_n,V)^{\GL(n,\Z)}.
\end{gather}
The second-named author \cite{KatadaIA} also made the following conjecture about the relation between the Albanese cohomology of $\IA_n$ and the twisted cohomology of $\Aut(F_n)$.
For any algebraic $\GL(n,\Q)$-representation $V$, let $H_A^i(\IA_n,V)= H_A^i(\IA_n,\Q)\otimes V$.

\begin{conjecture}[\cite{KatadaIA}]\label{conjkatada}
For any algebraic $\GL(n,\Q)$-representation $V$, the map $i^*$ is stably injective and its image stably equals to $H_A^*(\IA_n,V)^{\GL(n,\Z)}$. That is, we stably have an isomorphism of graded vector spaces
\begin{gather*}
   i^*: H^*(\Aut(F_n),V)\xrightarrow{\cong} H_A^*(\IA_n,V)^{\GL(n,\Z)}.
\end{gather*}
\end{conjecture}

Kawazumi and Vespa \cite{Kawazumi-Vespa} proposed a conjectural structure of the cohomology of $\Aut(F_n)$ with coefficients in $H^{p,q}$, the degree $1$ part of which we have mentioned in Remark \ref{remarkKV}, and which was previously known to hold for $p=0$ or $q=0$ by Djament--Vespa \cite{Djament-Vespa}, Djament \cite{Djament}, Vespa \cite{Vespa} and Randal-Williams \cite{Randal-Williams}.
Recently, Lindell \cite{Lindell22} proved their conjecture.

\begin{theorem}[Lindell \cite{Lindell22}]\label{conjKV}
Let $p,q\ge 0$.
If $i\neq p-q$, then we have
\begin{gather*}
    H^i(\Aut(F_n),H^{p,q})=0 
\end{gather*}
for sufficiently large $n$.
Otherwise, we have
\begin{gather*}
    H^{i}(\Aut(F_n),H^{q+i,q})=\calC_{\calP_0^{\circlearrowright}}(q+i,q)
\end{gather*}
for sufficiently large $n$.
(Here $\calC_{\calP_0^{\circlearrowright}}$ is the wheeled PROP associated to the wheeled completion of the operad $\calP_0$, see Remark \ref{remarkKV}.)
\end{theorem}

By combining Theorem \ref{conjKV} and the result of the second-named author \cite[Proposition 12.7]{KatadaIA}, it follows that Conjectures \ref{conjalb} and \ref{conjkatada} are equivalent.

\subsection{Conjectural stable structure of $H^*(\IA_n,\Q)$}

Here we propose a conjectural structure of the whole stable rational cohomology of $\IA_n$.

We have the restrictions of the cup product maps
$$\omega_{n,i}:\bigoplus_{j+k=i} H_A^j(\IA_n,\Q)\otimes H^k(\IA_n,\Q)^{\GL(n,\Z)}\to H^i(\IA_n,\Q),$$
which form a graded algebra map
$$\omega_{n}:H_A^*(\IA_n,\Q)\otimes H^*(\IA_n,\Q)^{\GL(n,\Z)}\to H^*(\IA_n,\Q).$$
The map $\omega_{n,i}$ is an isomorphism in degree $i=0,1$, since we have
$$
  H^i(\IA_n,\Q)=  H_A^i(\IA_n,\Q) \quad(i=0,1).
$$

\begin{conjecture}
\label{conjIAn}
The graded algebra map $\omega_n$ is a stable isomorphism.
\end{conjecture}

Conjectures \ref{conjIAn}, \ref{conjalb} and \ref{conjectureHIAinv} imply the following conjecture, which would give a complete algebraic structure for the stable rational cohomology of $\IA_n$.

\begin{conjecture}
\label{conjIAnW}
We stably have an isomorphism of graded algebras with $\GLnZ$-actions
\begin{gather*}
    H^*(\IA_n,\Q) \cong (W_*)^* \otimes 
    \Q[z_1,z_2,\dots],\quad \deg z_i=4i,
\end{gather*}
where $(W_*)^*$ denotes the graded dual of $W_*$.
\end{conjecture}

Conjecture \ref{conjIAn} is related to some other conjectures as follows.

\begin{theorem}\label{conjIAprop}
Suppose that Conjectures \ref{conjPi} and \ref{conjalb} hold.
Then Conjecture \ref{conjIAn} holds.
\end{theorem}

\begin{proof}
By the assumptions $(\SF_i)$, the graded algebra map
$$\omega_n: H_A^*(\IA_n,\Q)\otimes H^*(\IA_n,\Q)^{\GL(n,\Z)}\to H^*(\IA_n,\Q)$$
is a stable isomorphism if and only if
for any bipartition $\ul\lambda$, the cup product map
$$H_A^*(\IA_n,V_{\ulla})^{\GL(n,\Z)} \otimes H^*(\IA_n,\Q)^{\GL(n,\Z)}\xrightarrow{\cup} H^*(\IA_n,V_{\ulla})^{\GL(n,\Z)}$$
is a stable isomorphism.
Therefore, since Conjectures \ref{conjalb} and \ref{conjkatada} are equivalent, Conjecture \ref{conjIAn} holds by Conjecture \ref{conjkatada} and Lemma \ref{cohomIAstr} below.
\end{proof}

We have a morphism of graded vector spaces
$$
\xi: H^*(\Aut(F_n),V_{\ulla})\otimes H^*(\IA_n,\Q)^{\GL(n,\Z)} \to H^*(\IA_n,V_{\ulla})^{\GL(n,\Z)}
$$
defined by $\xi(u\ot v)=i^*(u)\cup v$ for $u\in H^*(\Aut(F_n),V_{\ulla})$ and $v\in H^*(\IA_n,\Q)^{\GL(n,\Z)}$, where $i^*$ is defined in \eqref{AutIA}.

\begin{lemma}\label{cohomIAstr}
Suppose that Conjecture \ref{conjPi} holds.
Then for each bipartition $\ulla$, the map $\xi$ is a stable isomorphism.
\end{lemma}

\begin{proof}
Let $\ul\lambda$ be a bipartition.
In what follows, we assume that $n$ is sufficiently large with respect to the cohomological degrees of cohomology groups that we consider.

We use induction on the cohomological degree $i$.
The case $i=0$ is trivial. Let $i\ge1$.
Suppose that we have 
\begin{gather}\label{IAdecomp}
    \xi: \bigoplus_{j+k=q} H^j(\Aut(F_n),V_{\ulla})\otimes H^k(\IA_n,\Q)^{\GL(n,\Z)}\xrightarrow{\cong} H^q(\IA_n,V_{\ulla})^{\GL(n,\Z)}
\end{gather}
for $0\le q\le i-1$.
We consider the Hochschild--Serre spectral sequence
$$
E(\ulla)_2^{p,q}
=H^{p}(\GL(n,\Z), H^{q}(\IA_n,V_{\ulla}))\Rightarrow H^{p+q}(\Aut(F_n),V_{\ulla}).
$$

If we have $H^j(\Aut(F_n),V_{\ulla})=0$ for any $j\le i-1$, then we have $H^q(\IA_n,V_{\ulla})^{\GL(n,\Z)}=0$ for any $q\le i-1$ by \eqref{IAdecomp}.
By the assumptions $(\SF_j)$ for $j\le i-1$ and Proposition \ref{stablefamilyBorel}, we have $E(\ulla)_2^{p,q}=0$ for any $q\le i-1$.
Therefore, by the convergence of the spectral sequence, 
one can check that $\xi=i^*: H^i(\Aut(F_n),V_{\ulla}) \xrightarrow{\cong} H^i(\IA_n,V_{\ulla})^{\GL(n,\Z)}$.

Otherwise, there is $0\le q_0\le i-1$ such that we have $H^{q_0}(\Aut(F_n),V_{\ulla})\neq 0$.
Then by the first half of Theorem \ref{conjKV}, we have $H^{q}(\Aut(F_n),V_{\ulla})= 0$ for $q\neq q_0$.
Therefore, by \eqref{IAdecomp}, $\xi$ gives the following isomorphisms 
\begin{gather}\label{IAullainv}
H^q(\IA_n,V_{\ulla})^{\GL(n,\Z)}\cong 
    \begin{cases}
    0 & (q<q_0)\\
    H^{q_0}(\Aut(F_n),V_{\ulla}) & (q=q_0)\\
    H^{q_0}(\IA_n,V_{\ulla})^{\GL(n,\Z)}\otimes H^{q-q_0}(\IA_n,\Q)^{\GL(n,\Z)} & (q_0<q\le i-1).
    \end{cases}
\end{gather}

Recall the Hochschild--Serre spectral sequence 
$E_2^{p,q}=H^{p}(\GL(n,\Z), H^{q}(\IA_n,\Q))$ that we used in the proof of Proposition \ref{IAnGLinv}.
Let $\hat{E}=\{\hat{E}_{r}\}_{r\ge2}$ denote the spectral sequence concentrated in bidegree $(0,q_0)$ with $\hat{E}_2^{0,q_0}=H^{q_0}(\IA_n,V_{\ulla})^{\GL(n,\Z)}$.
Then we have a canonical morphism of spectral sequences
$\psi:\hat{E}\to E(\ulla)$.
Consider the tensor product $\hat{E}\otimes E$ of two spectral sequences $\hat{E}$ and $E$.
We have a morphism of spectral sequences defined as the composition
$$\phi: \hat{E}\otimes E\xrightarrow{\psi\otimes\id}  E(\ulla)\otimes E\xrightarrow{\cup} E(\ulla).$$

We have the following commutative diagram
\begin{gather*}
    \xymatrix{
    H^{q_0}(\IA,V_{\ulla})^{\GL}\otimes H^{q-q_0}(\IA,\Q)^{\GL}\otimes H^p(\GL,\Q)  \ar[r]^-{\cup\otimes\id}\ar[d]^-{\id\otimes \cup} &
     H^{q}(\IA,V_{\ulla})^{\GL}\otimes H^p(\GL,\Q) 
     \ar[d]^-{\cup}\\
     H^{q_0}(\IA,V_{\ulla})^{\GL}\otimes  H^p(\GL, H^{q-q_0}(\IA,\Q)) \ar[r]^-{\cup=\phi_2^{p,q}}&
     H^p(\GL, H^{q}(\IA,V_{\ulla})),
    }
\end{gather*}
where we write $\GL=\GL(n,\Z)$ and $\IA=\IA_n$.
Here the vertical maps are isomorphisms for $q\le i-1$ by Proposition \ref{stablefamilyBorel}, and the upper horizontal map is an isomorphism for $q\le i-1$ by \eqref{IAullainv}.
Therefore, $\phi_2^{p,q}$ is also an isomorphism for $q\le i-1$.

By Lemma \ref{isomorphism-of-SS} below, the map $\phi_2^{0,i}: (\hat{E}\otimes E)_2^{0,i}\to  E(\ulla)_2^{0,i}$
is an isomorphism since we have
\begin{itemize}
\item $\phi_2^{p,q}$ is an isomorphism for $q\le i-1$,
\item for $j=i, i+1$, $(\hat{E} \otimes E)_\infty^{j} \cong \hat{E}_\infty^{0,q_0}\otimes E_\infty^{j-q_0}\cong
\hat{E}_2^{0,q_0} \otimes H^{j-q_0}(\Aut(F_n),\Q)=0$ by \cite{Galatius}, and
\item for $j=i,i+1$, $E(\ulla)_\infty^{j}\cong H^j(\Aut(F_n),V_{\ulla})=0$.
\end{itemize}
Therefore, we have
\begin{gather*}
\begin{split}
    \xi:&\bigoplus_{j+k=i} H^j(\Aut(F_n),V_{\ulla})\otimes H^k(\IA_n,\Q)^{\GL(n,\Z)}\\
    &\xrightarrow[\cong]{i^*\ot \id} 
    H^{q_0}(\IA_n,V_{\ulla})^{\GL(n,\Z)}\otimes H^{i-q_0}(\IA_n,\Q)^{\GL(n,\Z)}\\
    &\xrightarrow[\cong]{\cup=\phi_2^{0,i}} H^i(\IA_n,V_{\ulla})^{\GL(n,\Z)},
\end{split}
\end{gather*}
which completes the proof.
\end{proof}

\begin{lemma}
\label{isomorphism-of-SS}
Let $\phi:{}'E\to E$ be a morphism between first-quadrant spectral sequence.
Let $i\ge1$ be an integer.
Suppose that
\begin{itemize}
\item[(1)] $\phi_2^{p,q}:{}'E_2^{p,q}\to E_2^{p,q}$ is an isomorphism for $0\le q\le i-1$ and $0\le p\le 2i$.
\item[(2)]
${}'E_\infty^{p,q}=E_\infty^{p,q}=0$ for $p,q$ with $p+q=i,i+1$.
\end{itemize}
Then $\phi_2^{0,i}$ is an isomorphism.

\end{lemma}

\begin{proof}
Firstly, note that if $\phi_{r}^{p,q}$ and $\phi_{r}^{p+r,q-r+1}$ are isomorphisms, then we have
$$\phi_{r}^{p,q}(\ker {}'d_r^{p,q})=\ker d_r^{p,q},\quad \phi_{r}^{p+r,q-r+1}(\im {}'d_r^{p,q})=\im d_r^{p,q}.$$
Therefore, if all of the three maps $\phi_{r}^{p-r,q+r-1}, \phi_{r}^{p,q}, \phi_{r}^{p+r,q-r+1}$ are isomorphisms, then $\phi_{r+1}^{p,q}$ is also an isomorphism.
By the assumption $(1)$, it follows inductively that the maps $\phi_{r}^{p,q}$ are isomorphisms for any $q\le i-r+1$, $p\le 2i-2q$.

Let $\Lambda=\{(p,q)\mid p+q=i+1, p\ge 2\} \cup \{(0,i)\}$.
We will show, by a backward induction on $r$, that for $2\le r\le i+2$, the map $\phi_r^{p,q}$ is an isomorphism for each $(p,q)\in \Lambda$.
For each $(p,q)\in \Lambda$, the map $\phi_{i+2}^{p,q}$ is an isomorphism
since we have ${}'E_{\infty}^{p,q}(={}'E_{i+2}^{p,q})=0$ and $E_{\infty}^{p,q}(=E_{i+2}^{p,q})=0$ by the assumption $(2)$.
Suppose that $\phi_{r+1}^{p,q}$ is an isomorphism for each $(p,q)\in \Lambda$.
We consider the following commutative diagram
\begin{gather}\label{diag}
\xymatrix{
{}'E_{r}^{0,i}\ar[rr]^-{\phi_r^{0,i}}\ar[d]^{{}'d_r^{0,i}}&& E_{r}^{0,i}\ar[d]^{d_r^{0,i}}\\
{}'E_{r}^{r,i-r+1}\ar[rr]^{\phi_r^{r,i-r+1}}\ar[d]^{{}'d_r^{r,i-r+1}}&& E_{r}^{r,i-r+1}\ar[d]^{d_r^{r,i-r+1}}\\
{}'E_{r}^{2r,i-2r+2}\ar[rr]^-{\phi_r^{2r,i-2r+2}}&& E_{r}^{2r,i-2r+2}.
}
\end{gather}
As we mentioned in the first paragraph, $\phi_r^{r,i-r+1}$ and $\phi_r^{2r,i-2r+2}$ are isomorphisms.
Therefore, we have an isomorphism
$$\phi_r^{r,i-r+1}|_{\ker {}'d_{r}^{r,i-r+1}}: \ker {}'d_r^{r,i-r+1}\xto{\cong} \ker d_r^{r,i-r+1}.$$
By the induction hypothesis, $\phi_{r+1}^{r,i-r+1}$ is also an isomorphism, thus, we have 
$$
\ker {}'d_r^{r,i-r+1}/\im{}'d_{r}^{0,i} ={}'E_{r+1}^{r,i-r+1}\xto{\cong}  E_{r+1}^{r,i-r+1}= \ker d_r^{r,i-r+1}/\im d_{r}^{0,i}.
$$
Therefore, we have an isomorphism
$\phi_r^{r,i-r+1}|_{\im {}'d_{r}^{0,i}}: \im {}'d_r^{0,i}\xto{\cong} \im d_r^{0,i}$.
We also have 
$$
\ker {}'d_r^{0,i} ={}'E_{r+1}^{0,i}\xto{\cong}  E_{r+1}^{0,i}= \ker d_r^{0,i}
$$
since $\phi_{r+1}^{0,i}$ is also an isomorphism.
Hence, $\phi_{r}^{0,i}$ is also an isomorphism.
We can also check that $\phi_{r}^{p,q}$ are isomorphisms for $p+q=i+1, i-r+1<q\le i-1$ by considering the commutative diagram starting at ${}'E_r^{p,q}$ as in \eqref{diag}.

Therefore, by induction, $\phi_2^{0,i}$ is an isomorphism.
\end{proof}

\begin{remark}
In the statements of Theorem \ref{conjIAprop} and Lemma \ref{cohomIAstr}, we can replace the assumption, Conjecture \ref{conjPi}, with the assumption $(\mathrm{SA}_i)$, that we considered in Remark \ref{algverIAnGLinv}.
Consequently, if we assume Conjecture \ref{conjalb}, the following are equivalent:
\begin{itemize}
    \item Conjecture \ref{conjIAnW},
    \item Conjecture \ref{conjPi},
    \item for each $i\ge 0$, the hypothesis $(\mathrm{SA}_i)$ holds.
\end{itemize}
\end{remark}

\subsection{Types of irreducible components of $H^*(\IA_n,\Q)$}

Here we obtain some irreducible algebraic $\GLnQ$-representations which are not included in $H^*(\IA_n,\Q)$.
 
\begin{proposition}
Let $q\ge 0$ and suppose that $(\SF_i)$ holds for $i\le q$ so that we have
\begin{gather*}
    H^q(\IA_n,\Q)\cong \bigoplus_{\ul\lambda} V_{\ul\lambda}^{\oplus c(q,\ul\lambda)}
\end{gather*}
for integers $c(q,\ul\lambda)\ge0$ for all sufficiently large $n$.
Then we have
$c(q,\ul\lambda)=0$ for bipartitions $\ul\lambda$ such that either $\deg\ul\lambda\le -q-1$ or $0\le \deg\ul\lambda$ with $\ul\lambda\neq (0,0)$.
\end{proposition}

\begin{proof}
Let $\ul\lambda$ be a bipartition such that either $\deg\ul\lambda\le -q-1$ or $0\le \deg\ul\lambda$ with $\ul\lambda\neq (0,0)$.
By Theorem \ref{conjKV}, for any $0\le i\le q$, we have
\begin{equation}\label{Autvanish}
    H^i(\Aut(F_n),V_{\ul\lambda^*})=0
\end{equation}
for sufficiently large $n$.
We will prove by induction on $i$ that for any $0\le i\le q$, we have
$$H^i(\IA_n,V_{\ul\lambda^*})^{\GLnZ}=0$$
for sufficiently large $n$.
For $i=0$, we have $H^0(\IA_n,V_{\ul\lambda^*})^{\GLnZ}=V_{\ul\lambda^*}^{\GLnZ}=0$.
Suppose that for any $k\le i-1$, we have $H^k(\IA_n,V_{\ul\lambda^*})^{\GLnZ}=0$ for sufficiently large $n$.
Then by the assumption $(\SF_i)$ and by Proposition \ref{stablefamilyBorel}, for any $j\le i+1$, we have $E_2^{j,k}=H^j(\GL(n,\Z),H^k(\IA_n,V_{\ul\lambda^*}))=0$
for sufficiently large $n$.
Therefore, by \eqref{Autvanish}, we have $H^{i}(\IA_n,V_{\ul\lambda^*})^{\GLnZ}=0$, which implies that $V_{\ul\lambda}$ is not included in $H^q(\IA_n,\Q)$.
\end{proof}

\section{Stable rational cohomology of $\IO_n$}\label{secstrIO}
In this section, we study the stable rational cohomology of $\IO_n$ as in the case of $\IA_n$.

\subsection{$\GL(n,\Z)$-invariant part of $H^*(\IO_n,\Q)$}

We make the following conjecture about the $\GL(n,\Z)$-invariant part of $H^*(\IO_n,\Q)$.

\begin{conjecture}\label{conjectureHIOinv}
We stably have an isomorphism of graded algebras
$$H^*(\IO_n,\Q)^{\GL(n,\Z)}\cong \Q[z_1,z_2,\ldots],\quad \deg z_i=4i.$$
\end{conjecture}

We will prove Conjecture \ref{conjectureHIOinv} by assuming the following conjecture, which is analogous to Conjecture \ref{conjPi}.

\begin{conjecture}\label{conjSFIO}
For each $i\ge 0$, the following hypothesis $(\SF'_i)$ holds.
\begin{itemize}
    \item[$(\SF'_i)$]$\{H^i(\IO_n,\Q)\}_n$ is a stable family of $\GL(n,\Z)$-representations.
\end{itemize}
\end{conjecture}

It is known that $(\SF'_i)$ holds for $i=0,1$ \cite{Kawazumi}.

\begin{remark}
Similarly to the case of $\IA_n$, we can replace the assumption in Theorems \ref{HIOinv} and \ref{theoremIO}, i.e. Conjecture \ref{conjSFIO}, by the weaker assumption that $\{H^i(\IO_n,\Q)\}_n$ is a stably algebraic family.
\end{remark}

\begin{theorem}\label{HIOinv}
Conjecture \ref{conjSFIO} implies Conjecture \ref{conjectureHIOinv}.
\end{theorem}

\begin{proof}
We have an exact sequence 
$$1\to \IO_n\to \Out(F_n)\to \GL(n,\Z)\to 1$$
and by Galatius \cite{Galatius}, we have $H^i(\Out(F_n),\Q)=0$ for $n$ sufficiently large with respect to $i$.
Therefore, the proof is similar to that of Proposition \ref{IAnGLinv}.
\end{proof}

\subsection{Conjectural stable structure of $H^*(\IO_n,\Q)$}

The whole structure of the stable rational cohomology of $\IO_n$ has not been determined in degree $\ge 2$.
We propose a conjectural stable structure of $H^*(\IO_n,\Q)$ as in the case of $\IA_n$.

We have the restriction of the cup product maps
$$
 \omega'_{n,i}: \bigoplus_{j+k=i}H_A^j(\IO_n,\Q)\otimes H^k(\IO_n,\Q)^{\GL(n,\Z)}\to H^i(\IO_n,\Q),
$$
which form a graded algebra map 
$$
 \omega'_{n}: H_A^*(\IO_n,\Q)\otimes H^*(\IO_n,\Q)^{\GL(n,\Z)}\to H^*(\IO_n,\Q).
$$
The map $\omega'_{n,i}$ is an isomorphism in degree $i=0,1$.

\begin{conjecture}\label{omega'isom}
The graded algebra map $\omega'_{n}$ is a stable isomorphism.
\end{conjecture}

The second-named author \cite{KatadaIA} also defined the traceless part $W^O_*$ of the graded symmetric algebra $S^*(U^O_*)$ of the graded $\GL(n,\Q)$-representation $U^O_*=\bigoplus_{i\ge 1}U^O_i$, where $U^O_i=U_i$ for $i\ge 2$ and $U^O_1=U_1/H\cong V_{1^2,1}$, and proposed a conjectural structure of the stable Albanese homology $H^A_*(\IO_n,\Q)$.

\begin{conjecture}[\cite{KatadaIA}]\label{conjalbIO}
Let $i\ge 0$.
We have a $\GL(n,\Q)$-isomorphism
$$
  H^A_i(\IO_n,\Q)\cong W^O_i
$$
for sufficiently large $n$.
\end{conjecture}

By combining Conjecture \ref{conjalbIO} with Conjectures \ref{conjectureHIOinv} and \ref{omega'isom}, we obtain the following conjecture, which would give a complete algebraic structure of the stable rational cohomology of $\IO_n$.

\begin{conjecture}\label{conjIOnWO}
We stably have an isomorphism of graded algebras with $\GL(n,\Z)$-actions
$$H^*(\IO_n,\Q)\cong (W^O_*)^*\otimes 
    \Q[z_1,z_2,\dots],\quad \deg z_i=4i.
$$
\end{conjecture}

\subsection{Twisted stable cohomology of $\Out(F_n)$}

The stable cohomology of $\Out(F_n)$ with coefficients in $H^{p,0}$ and $H^{0,q}$ has been determined as follows.

\begin{theorem}[Randal-Williams \cite{Randal-Williams}]\label{Outtwisted}
We have
\begin{gather*}
    \begin{split}
        H^i(\Out(F_n),(H_{\Z}^*)^{\otimes q})&=0 \quad \text{for}\quad 2i\le n-q-3,\\
        H^i(\Out(F_n),H^{\otimes q})&=0 \quad \text{for}\quad 2i\le n-q-3,\;i\neq q,
    \end{split}
\end{gather*}
and for $n\ge 4q+3$,
$H^q(\Out(F_n),H^{\otimes q})\otimes \sgn_{\gpS_q}$ is the permutation module on the set of partitions of $\{1,\dots,q\}$ having no parts of size $1$.
\end{theorem}

Here we make the following conjecture about the relation between the twisted stable cohomology of $\Out(F_n)$ and the stable Albanese cohomology of $\Out(F_n)$, which is an analogue of Conjecture \ref{conjkatada}.

\begin{conjecture}\label{conjIOOut}
Let $V$ be an algebraic $\GL(n,\Q)$-representation.
We stably have an isomorphism of graded algebras
$$H^*(\Out(F_n),V)\cong H_A^*(\IO_n,V)^{\GL(n,\Z)}.$$
\end{conjecture}

Conjectures \ref{conjIOOut} and \ref{conjalbIO} imply the following.

\begin{conjecture}\label{conjOut}
Let $i\ge 0$ and let $\ulla$ be a bipartition such that $i\neq \deg \ulla$.
Then we have
$$H^i(\Out(F_n),V_{\ul\lambda})=0$$
for sufficiently large $n$.
\end{conjecture}

We obtain the following proposition which is analogous to Theorem \ref{conjIAprop}.

\begin{theorem}
\label{theoremIO}
Conjectures \ref{conjSFIO}, \ref{conjIOOut} and \ref{conjOut} imply Conjecture \ref{omega'isom}.
\end{theorem}

\subsection{Relation between $H^*(\IO_n,\Q)$ and $H^*(\IA_n,\Q)$}

Here we study the relation between $H^*(\IO_n,\Q)$ and $H^*(\IA_n,\Q)$.

Let $n\ge2$. The projection $p:\IA_n\twoheadrightarrow \IO_n$ induces a morphism of graded algebras with $\GL(n,\Z)$-actions
$p^*:H^*(\IO_n,\Q)\to H^*(\IA_n,\Q)$.
The morphism of graded $\GLnZ$-representations $H^*(\IA_n,\Q)\to H^*(F_n,\Q)$ induced by the inclusion $F_n\hookrightarrow \IA_n$ has a splitting 
$$
t: H^*(F_n,\Q)\to H^*(\IA_n,\Q)
$$
since $H^*(F_n,\Q)=\Q[0]\oplus H^*[1]$ and $H^1(\IA_n,\Q)\cong (U/H)^*\oplus H^*=V_{1,1^2}\oplus V_{0,1}$.
Define a morphism of graded $\GL(n,\Z)$-representations
$$
\psi: H^*(\IO_n,\Q)\otimes H^*(F_n,\Q)\to H^*(\IA_n,\Q)
$$
by $\psi(u\otimes v)=p^*(u) \cup t(v)$.

\begin{proposition}\label{IAandIO}
Let $n\ge2$.
We have an isomorphism of graded $\GL(n,\Z)$-representations
$$
\gr H^*(\IA_n,\Q)\cong H^*(\IO_n,\Q)\otimes H^*(F_n,\Q)
$$
where the left hand side denotes the associated graded of $H^*(\IA_n,\Q)$ with respect to a filtration (of length 2) of the $\GL(n,\Z)$-representation $H^*(\IA_n,\Q)$.
\end{proposition}

\begin{proof}
Consider the Hochschild--Serre spectral sequence 
\begin{gather*}
    E_2^{p,q}=H^p(\IO_n,H^q(F_n,\Q))\Rightarrow H^{p+q}(\IA_n,\Q)
\end{gather*}
for the short exact sequence $1\to F_n\to \IA_n\to \IO_n\to 1$.
Since $\IO_n$ acts trivially on $H^*(F_n,\Q)$ we have
\begin{gather*}
    E_2^{p,q}=H^p(\IO_n,\Q)\otimes H^q(F_n,\Q).
\end{gather*}
For $q>1$, we have $E_2^{p,q}=0$ since $H^q(F_n,\Q)=0$.
We have $d_2^{0,1}=0:E_2^{0,1}=H^1(F_n,\Q)\to E_2^{2,0}=H^2(\IO_n,\Q)$ since otherwise it would follow that $H^1(\IA_n,\Q)\cong E_2^{1,0}=H^1(\IO_n,\Q)$.
By multiplicativity of the spectral sequence, we see that $d_2^{*,*}=0$.
Hence $E_\infty=E_2$.
Therefore the result follows.
\end{proof}

Conjectures \ref{conjPi} and \ref{conjSFIO} and Proposition \ref{IAandIO} imply the following conjecture.

\begin{conjecture}\label{conjstrIOIA}
The map $\psi$ is a stable isomorphism of graded algebraic $\GL(n,\Z)$-representations.
\end{conjecture}

Note that for $n\ge 2$, we have an isomorphism of graded algebraic $\GL(n,\Z)$-representations
$$
 \psi_A:  H_A^*(\IO_n,\Q)\otimes H_A^*(F_n,\Q)\xrightarrow{\cong}H_A^*(\IA_n,\Q)
$$ \cite[Proposition 9.8]{KatadaIA}.
The projection $p:\IA_n\twoheadrightarrow \IO_n$ induces a morphism of graded algebras $p^*: H^*(\IO_n,\Q)^{\GLnZ}\xrightarrow{\cong} H^*(\IA_n,\Q)^{\GLnZ}$.
Then we have the following commutative diagram
\begin{gather*}
    \xymatrix{
    H_A^*(\IO_n,\Q)\otimes H_A^*(F_n,\Q)\otimes \Q[y_1,y_2,\ldots]\ar[r]^-{\psi_A}_-{\cong}\ar[d]_-{\id\otimes\id\otimes \varphi}  &
    H_A^*(\IA_n,\Q)\otimes \Q[y_1,y_2,\ldots]\ar[d]^-{\id\otimes \phi}\\
    H_A^*(\IO_n,\Q)\otimes H_A^*(F_n,\Q)\otimes H^*(\IO_n,\Q)^{\GLnZ}\ar[r]^-{\psi_A\otimes p^*}\ar[d]_-{\omega'_n\otimes \id} &
    H_A^*(\IA_n,\Q)\otimes H^*(\IA_n,\Q)^{\GLnZ}\ar[d]^-{\omega_n}\\
    H^*(\IO_n,\Q)\otimes H^*(F_n,\Q)\ar[r]^-{\psi} &  H^*(\IA_n,\Q),
    }
\end{gather*}
where $\phi$ (resp. $\varphi$) is the canonical map that sends $y_i$ to $y_i\in H^{4i}(\IA_n,\Q)^{\GLnZ}$ (resp. $y_i\in H^{4i}(\IO_n,\Q)^{\GLnZ}$).
Note that we have conjectured that the vertical maps are all stable isomorphisms (see Proposition \ref{IAnGLinv}, Theorem \ref{HIOinv} and Conjectures \ref{conjIAn} and \ref{omega'isom}).
By Conjecture \ref{conjstrIOIA} and the fact that $\psi_A$ is an isomorphism, it is natural to make the following conjecture as well.

\begin{conjecture}\label{p*}
    The morphisms
    $p^*: H^i(\IO_n,\Q)^{\GLnZ}\to H^i(\IA_n,\Q)^{\GLnZ}$
    are isomorphisms for $n$ sufficiently large with respect to $i$.
\end{conjecture}

\begin{question}
    For coefficients in $\Z$, are the morphisms
    $p^*: H^i(\IO_n,\Z)^{\GLnZ}\to H^i(\IA_n,\Z)^{\GLnZ}$
    always isomorphisms?
\end{question}

\section{Conjectures about the stable cohomology of the Torelli groups}
\label{secConjectures}

In this section, we discuss some conjectures about the  stable rational cohomology of the Torelli groups and the relation to the stable rational cohomology of the IA-automorphism groups of free groups.

\subsection{Stable cohomology of the mapping class groups and the symplectic groups}\label{subsecmcgsp}

Let $\Sigma_{g}$ (resp. $\Sigma_{g,1}$) denote a compact oriented surface of genus $g$ (resp. with one boundary component).
Let $\calM_{g}$ (resp. $\calM_{g,1}$) denote the \emph{mapping class group} of $\Sigma_{g}$ (resp. $\Sigma_{g,1}$).
Let $\calI_{g}$ (resp. $\calI_{g,1}$) denote the \emph{Torelli group}, which is defined as the kernel of the canonical surjective homomorphism from $\calM_{g}$ (resp. $\calM_{g,1}$) to the symplectic group $\Sp(2g,\Z)$. Then we have exact sequences
\begin{gather}\label{Spexact}
   1\to \calI_{g}\to \calM_{g}\to \Sp(2g,\Z)\to 1,\quad  1\to \calI_{g,1}\to \calM_{g,1}\to \Sp(2g,\Z)\to 1.
\end{gather}
Then $\Sp(2g,\Z)$ acts on the cohomology of the Torelli groups.
See e.g. \cite{Morita-survey} for more precise definitions.

Madsen and Weiss \cite{Madsen-Weiss} proved the Mumford conjecture \cite{Mumford}, which determines the stable rational cohomology of the mapping class groups.

\begin{theorem}[Madsen--Weiss \cite{Madsen-Weiss}]\label{madsenweiss}
We stably have an algebra isomorphism
 $$H^{*}(\calM_{g}, \Q)\cong H^{*}(\calM_{g,1}, \Q) \cong \Q[e_1,e_2,\ldots],
 $$
where $e_i$ is the Mumford--Morita--Miller class of degree $2i$ \cite{Mumford,Morita1987,Miller}.
\end{theorem}

Borel's stability and vanishing theorem can also be applied to $\Sp(2g,\Z)$.
In the case of $\Sp(2g,\Z)$, the stable range given by Li and Sun \cite{Li-Sun} is equal to Borel's stable range that Borel remarked in \cite{Borel2}, which was computed by Tshishiku \cite{Tshishiku}.

\begin{theorem}[Borel \cite{Borel1, Borel2}, Li--Sun \cite{Li-Sun}, Tshishiku \cite{Tshishiku}]\label{SpBorel}
(1) We have an algebra isomorphism
$$H^{*}(\Sp(2g,\Z), \Q) \cong \Q[u_1,u_2,\ldots], \quad \deg u_i=4i-2$$
in $*\le g-1$.
 
(2) Let $V$ be an algebraic $\Sp(2g,\Q)$-representation such that $V^{\Sp(2g,\Q)}=0$.
Then we have 
$$
H^{p}(\Sp(2g,\Z), V)=0 \quad\text{for } p\le g-1.
$$
\end{theorem}

Here we recall some relations between the stable rational cohomology of $\Sp(2g,\Z)$ and that of $\calM_g$ described in the introduction of Morita's paper \cite{Morita1996}.
The surjective map $p:\calM_{g}\to \Sp(2g,\Z)$ induces an injective map
\begin{gather}
\label{p-star}
p^*:\varprojlim_g H^*(\Sp(2g,\Z),\Q)\hookrightarrow \varprojlim_g H^*(\calM_{g},\Q).
\end{gather}
The image of $\varprojlim_g H^*(\Sp(2g,\Z),\Q)$ under $p^*$ is $\Q[e_1,e_3,\ldots]$.
For $i\ge1$ and $g$ sufficiently large with respect to $i$, let $e'_{2i-1}\in H^{4i-2}(\Sp(2g,\Z),\Q)$ denote the element such that $p^*(e'_{2i-1})=e_{2i-1}$.

\subsection{$\Sp(2g,\Z)$-invariant part of the stable rational cohomology of the Torelli groups}

Kawazumi and Morita made the following conjecture about a stable structure of the $\Sp(2g,\Z)$-invariant part of the stable rational cohomology of $\calI_{g}$ \cite[Conjecture 13.8]{Kawazumi-Morita} (see also \cite[Conjecture 3.4]{Morita-survey}).

\begin{conjecture}[Kawazumi--Morita \cite{Kawazumi-Morita}]\label{conjectureinvKM}
We stably have 
\begin{gather*}
H^*(\calI_{g},\Q)^{\Sp(2g,\Z)}\cong \Q[e_2,e_4,\dots],
\end{gather*}
where $e_{2i}\in H^{4i}(\calI_{g},\Q)^{\Sp(2g,\Z)}$ is the image of $e_{2i}$ under the map induced by the inclusion $\calI_{g}\hookrightarrow \calM_{g}$.
\end{conjecture}

It is natural to make the following conjecture, which is a variant of Conjecture \ref{conjectureinvKM}.

\begin{conjecture}\label{conjectureHIginv}
The $\Sp(2g,\Z)$-invariant part of the rational cohomology of $\calI_{g,1}$ stabilizes, that is, 
for each $i\ge0$, there is an integer $N\ge0$ such that for all $g\ge N$ the inclusion map $\calI_{g,1}\hookrightarrow\calI_{g+1,1}$ induces an isomorphism
\begin{gather*}
H^i(\calI_{g+1,1},\Q)^{\Sp(2g+2,\Z)}\congto
H^i(\calI_{g,1},\Q)^{\Sp(2g,\Z)}.
\end{gather*}
Moreover, we stably have an algebra isomorphism
\begin{gather*}
H^*(\calI_{g,1},\Q)^{\Sp(2g,\Z)}\cong \Q[e_2,e_4,\dots],
\end{gather*}
where $e_{2i}
\in H^{4i}(\calI_{g,1},\Q)^{\Sp(2g,\Z)}$ is the image of $e_{2i}$ under the map induced by the inclusion $\calI_{g,1}\hookrightarrow \calM_{g,1}$.
\end{conjecture}

For the second cohomology, we have the following Torelli group version of Corollary \ref{DayPutmanrational}.
This result in some stable range has probably been known to some experts regarding Conjecture \ref{conjectureinvKM}, but we include it here because we could not find such a result in the literature.

\begin{proposition}
    For $g\ge 4$, we have 
    \begin{gather*}
        H^2(\calI_{g,1},\Q)^{\Sp(2g,\Z)}=H^2(\calI_{g},\Q)^{\Sp(2g,\Z)}=0.
    \end{gather*}
\end{proposition}
\begin{proof}
    Associated to the exact sequences \eqref{Spexact}, we have the Hochschild--Serre spectral sequences 
    \begin{gather*}
E_2^{p,q}=H^p(\Sp(2g,\Z),H^q(\calI_{g},\Q))\Rightarrow H^{p+q}(\calM_{g},\Q),\\
E_2^{p,q}=H^p(\Sp(2g,\Z),H^q(\calI_{g,1},\Q))\Rightarrow H^{p+q}(\calM_{g,1},\Q),    
    \end{gather*}
    respectively.
    Since we have $E_2^{2,0}=\Q$, $E_2^{1,1}=E_2^{3,0}=E_2^{2,1}=0$ by Theorem \ref{SpBorel}, we have $E_2^{0,2}=0$ by Theorem \ref{madsenweiss}.
    The stable range of the cohomology of $\calM_{g}$ and $\calM_{g,1}$ are given by Harer \cite{Harer} (see also \cite[Theorem 5.8]{Farb-Margalit}) and the stable range of the cohomology of $\Sp(2g,\Z)$ is given by Tshishiku \cite{Tshishiku} (see Theorem \ref{SpBorel}).
\end{proof}

Similarly to Definition \ref{stablefamily} for $\GLnZ$-representations, we make the following definition.
\begin{definition}
\label{stablefamilySp}
By a \emph{stable family of $\Sp(2g,\Z)$-representations}, we mean a family $W_*=\{W_g\}_{g\ge0}$ of $\Sp(2g,\Z)$-representations $W_g$ such that there are finitely many partitions $\lambda_i$ (possibly with repetitions) such that, for sufficiently large $g$,  we have an $\Sp(2g,\Z)$-module isomorphism
\begin{gather*}
W_g \cong \bigoplus_{i} V^{\Sp}_{\la_i}(g),
\end{gather*}
where for a partition $\lambda$ (with $l(\lambda)\le g$), $V^{\Sp}_{\la}(g)$ denotes the irreducible $\Sp(2g,\Z)$-representation corresponding to $\lambda$. \end{definition}

The following conjecture would follow from \cite[Conjectures 6.1, 6.6 and 6.7]{Church-Farb}, and is analogous to Conjecture \ref{conjPi}.

\begin{conjecture}[Church--Farb \cite{Church-Farb}]
\label{conjSFiI}
For each $i\ge0$, the following hypothesis  $(\SF_i^\calI)$ holds. \begin{itemize}
    \item[($\SF_i^\calI$)] $\{H^i(\calI_{g,1},\Q)\}_g$ is a stable family of $\Sp(2g,\Z)$-representations.
\end{itemize}
\end{conjecture}
We have ($\SF_i^\calI$) for $i=0,1$. The cases $i\ge2$ are not known.

We will prove the following theorem which is analogous to Theorem \ref{thmHIAinv}. Morita \cite{Morita-survey} gave a related argument.

\begin{theorem}
Conjecture \ref{conjSFiI} implies Conjecture \ref{conjectureHIginv}.
Similarly, the variant of Conjecture \ref{conjSFiI} for $\{H^i(\calI_{g},\Q)\}_g$ implies Conjecture \ref{conjectureinvKM}.
\end{theorem}

\begin{proof}
We consider only the case of $\calI_{g,1}$; the case of $\calI_g$ can be proved in the same way.

We have the Hochschild--Serre spectral sequence $$E_2^{p,q}(g)=H^p(\Sp(2g,\Z),H^q(\calI_{g,1},\Q))\Rightarrow H^{p+q}(\calM_{g,1},\Q)$$ 
associated to the exact sequence \eqref{Spexact}.

The map induced by the inclusion $\calI_{g,1}\hookrightarrow \calM_{g,1}$ can be decomposed into the following maps
$$\phi: H^{4j}(\calM_{g,1},\Q)\twoheadrightarrow  E_{\infty}^{0,4j}(g)\hookrightarrow E_2^{0,4j}(g)=H^{4j}(\calI_{g,1},\Q)^{\Sp(2g,\Z)}.$$
Let $e'_{2i}=\phi(e_{2i})\in H^{4j}(\calI_{g,1},\Q)^{\Sp(2g,\Z)}$.

Let $\Q[\wti{e}_1, \wti{e}_2,\ldots]$ be a bigraded polynomial algebra, where $\deg \wti{e}_{2i-1}=(4i-2,0)$ and $\deg \wti{e}_{2i}=(0,4i)$ for each $i\ge 1$.
Consider the following hypothesis
\begin{itemize}
    \item[$(P_{m,n})$] There is $g_{m,n}\ge 0$ such that for all $g\ge g_{m,n}$, the algebra map
    $$f_{m,n}:\Q[\wti e_1,\wti e_3,\dots, \wti e_{2[\frac{m+2}{4}]-1}, \wti e_2,\wti e_4,\ldots , \wti e_{2[\frac{n}{4}]}]\to E_2^{*,*}(g)$$
    which maps each $\wti e_{i}$ to $e'_{i}$,
    is an isomorphism in bidegree $(p,q)$ with $0\le p\le m,\; 0\le q\le n$, and 
     we have isomorphisms
    $$E_2^{p,q}(g)\cong E_{\infty}^{p,q}(g)$$
    for any $0\le p\le m,\; 0\le q\le n$.
\end{itemize}
Note that we have $E_2^{p,q}(g)\cong E_{\infty}^{p,q}(g)$ 
if and only if the differentials to and from $E_r^{p,q}$ are zero for all $r\ge2$.

We have $(P_{m,0})$ for any $m\ge0$ since we have $p^*(e'_{2i-1})=e_{2i-1}$ for any $i\ge 1$ and we have 
$\varprojlim_g p^*(H^*(\Sp(2g,\Z),\Q))= \Q[e_1,e_3,\ldots]$, where $p^*$ is the injective map in \eqref{p-star} with $\calM_g$ replaced by $\calM_{g,1}$.

We prove $(P_{0,n})$ by induction on $n$.
For $n=0$, we have already seen that $(P_{0,0})$ holds.
Let $n\ge0$ and suppose that $(P_{0,n})$ holds.
We will first show that $(P_{m,n})$ holds for any $m\ge0$.
Let $0\le p\le m$, $0\le q\le n$ and $g$ sufficiently large with respect to $m$ and $n$.
By the assumption ($\SF_q^\calI$) and Theorem \ref{SpBorel},
we have 
$$
 E_2^{p,q}(g)\cong E_2^{p,0}(g)\otimes E_2^{0,q}(g),
$$
and thus 
$f_{m,n}$ in bidegree $(p,q)$ can be identified with
$f_{m,0}\otimes f_{0,n}: \Q[\wti e_1,\wti e_3,\dots]^{p,0}\otimes\Q[\wti e_2,\wti e_4,\dots]^{0,q}\to E_2^{p,0}(g)\otimes E_2^{0,q}(g)$,
which is an isomorphism by the hypotheses $(P_{m,0})$ and $(P_{0,n})$.
Moreover, by $(P_{m,0})$ and $(P_{0,n})$, we have
$$
 E_2^{p,q}(g)\cong E_2^{p,0}(g)\otimes E_2^{0,q}(g)\cong E_{\infty}^{p,0}(g)\otimes E_{\infty}^{0,q}(g).
$$
By the definition of the Hochschild--Serre spectral sequence, $E_{\infty}^{p,q}(g)$ is a subquotient of $E_2^{p,q}(g)$.
Since $\varprojlim_g H^*(\calM_{g,1},\Q)$ has no relation for the generators $e_i$ by Theorem \ref{madsenweiss}, and since the multiplication map of $E_{\infty}^{*,*}(g)$ is compatible with the multiplication map of $H^*(\calM_{g,1},\Q)$, it follows that if $g$ is large enough, then
the bigraded algebra $E_\infty^{*,*}(g)$ has no relation in bidegree $(p,q)$ with $0\le p\le m,\; 0\le q\le n$,
and thus we have
$E_{\infty}^{p,q}(g)\cong E_2^{p,q}(g)$.
Therefore, $(P_{m,n})$ holds.

We will next show that $(P_{0,n+1})$ holds.
Since the Hochschild--Serre spectral sequence is first-quadrant, the differentials to $E_r^{0,n+1}(g)$ are zero.
The differentials from $E_r^{0,n+1}(g)$ are also zero since $(P_{n+2,n})$ holds.
Therefore, we have $E_2^{0,n+1}(g)\cong E_{\infty}^{0,n+1}(g)$.
By the hypothesis $(P_{n+2,n})$, we have $E_{\infty}^{p,q}(g)\cong \Q[\wti{e}_1,\wti{e}_2,\ldots]^{p,q}$ for $p+q=n+1$, $1\le p\le n+1$. 
By the argument in the above paragraph,
$E_{\infty}^{0,n+1}(g)$ includes $\Q[e'_2,\ldots, e'_{2[\frac{n}{4}]}]^{0,n+1}$.
Therefore, by using Theorem \ref{madsenweiss}, we have 
$$\dim E_2^{0,n+1}(g)- \dim \Q[e'_2,\ldots, e'_{2[\frac{n}{4}]}]^{0,n+1}
=
\begin{cases}
1 & (n+1\equiv0\pmod4)\\
0 & (\text{otherwise}).
\end{cases}
$$
Since $E_2^{0,n+1}(g)\cong E_{\infty}^{0,n+1}(g)$, the map $\phi: H^{n+1}(\calM_{g,1},\Q)\to H^{n+1}(\calI_{g,1},\Q)^{\Sp(2g,\Z)}$ is surjective.
Therefore, by dimension counting, we have $e'_{\frac{n+1}{2}}\neq 0$ if $n+1\equiv0\pmod4$ and thus the algebra map $f_{0,n+1}$ is an isomorphism in bidegree $(0,q)$ for $q\le n+1$.
This implies $(P_{0,n+1})$.

We now have $(P_{0,n})$ for all $n\ge0$. 
Thus we stably have an algebra isomorphism
\begin{gather*}
H^*(\calI_{g,1},\Q)^{\Sp(2g,\Z)}\cong \Q[e'_2,e'_4,\dots].
\end{gather*}
\end{proof}

\begin{remark}
    Similarly to the case of $\IA_n$, we can replace the assumption, Conjecture \ref{conjSFiI}, with a weaker assumption that $\{H^p(\calI_{g,1},\Q)\}_g$ is a stably algebraic family of $\Sp(2g,\Z)$-representations.
\end{remark}

The Dehn--Nielsen--Baer embedding maps the mapping class group $\calM_{g,1}$ into $\Aut(F_{2g})$ 
and restricts to 
an embedding $i_{g,1}:\calI_{g,1}\to\IA_{2g}$.
Then Conjectures \ref{conjectureHIAinv} and \ref{conjectureHIginv} imply the following conjecture.

\begin{conjecture}\label{conjIandIAinv}
The graded algebra homomorphism
$$i_{g,1}^*: H^*(\IA_{2g},\Q)^{\GL(2g,\Z)} \to H^*(\calI_{g,1},\Q)^{\Sp(2g,\Z)}$$ 
induced by $i_{g,1}$
is a stable isomorphism.
\end{conjecture}

\begin{remark}
In Section \ref{antiAut}, we defined $y_k\in H^{4k}(\IA_{2g},\Q)$, which is expected to be a nonzero scalar multiple of $T\beta_{2k+1}^0\in H^{4k}(\IA_{2g},\R)$.
Morita, Sakasai and Suzuki \cite{Morita-Sakasai-Suzuki} proved that $i_{g,1}^*(T\beta_{2k+1}^0)$ is a nonzero scalar multiple of $e_{2i}\in H^{4i}(\calI_{g,1},\Q)^{\Sp(2g,\Z)}$.
Therefore, Conjecture \ref{conjIandIAinv} may follow from Proposition \ref{IAnGLinv} and Conjecture \ref{conjectureHIginv}.
\end{remark}

\subsection{The stable rational cohomology of the Torelli groups}

Here we make a conjecture about the relation between the stable rational cohomology and the stable Albanese cohomology of the Torelli groups $\calI_g$ and $\calI_{g,1}$.

The stable Albanese cohomology of the Torelli groups of degree $\le3$ has been determined by Johnson \cite{Johnson}, Hain \cite{Hain}, Sakasai \cite{Sakasai} and Kupers--Randal-Williams \cite{KRW}.
Lindell \cite{Lindell} detected subquotient $\Sp(2g,\Z)$-representations of the higher degree stable Albanese cohomology of the Torelli groups.

By \cite{Hainsurvey,KRW21} it follows that the stable Albanese cohomology of the Torelli groups is a quotient of the cohomology of the associated graded Lie algebra of the Torelli groups with respect to the lower central series.
The latter cohomology has been studied by Hain \cite{Hain}, Garoufalidis--Getzler \cite{GG} and 
 Kupers--Randal-Williams \cite{KRW21}.

Recently, 
the second-named author \cite{KatadaIA} proposed conjectural structures of the stable Albanese cohomology of the Torelli groups in a way similar to the Albanese cohomology of $\IA_n$, which is true in degree $\le 3$ by \cite{Hain, Sakasai, KRW}.
In particular, she made the following conjecture about the $\Sp(2g,\Z)$-invariant part.

\begin{conjecture}[\cite{KatadaIA}]
We stably have isomorphisms of graded algebras
\begin{gather*}
H_A^*(\calI_{g},\Q)^{\Sp(2g,\Z)}\cong H_A^*(\calI_{g,1},\Q)^{\Sp(2g,\Z)}\cong \Q[z_1,z_2,\dots],
\end{gather*}
where $\deg z_i=4i$.
\end{conjecture}

By the definition of the Albanese cohomology, we have inclusion maps
$$\iota_{g}^{i}: H_A^i(\calI_{g},\Q)\hookrightarrow H^i(\calI_{g},\Q),\quad \iota_{g,1}^{i}: H_A^i(\calI_{g,1},\Q)\hookrightarrow H^i(\calI_{g,1},\Q),$$
which form graded algebra maps
$$\iota_{g}: H_A^*(\calI_{g},\Q)\hookrightarrow H^*(\calI_{g},\Q),\quad
\iota_{g,1}: H_A^*(\calI_{g,1},\Q)\hookrightarrow H^*(\calI_{g,1},\Q).$$

The maps $\iota_{g}^{i}$ and $\iota_{g,1}^{i}$ are isomorphisms in degree $i=0,1$.
It seems natural to make the following conjecture, which says that the rational cohomology and the Albanese cohomology of the Torelli groups stably coincide.

\begin{conjecture}\label{conjIgstr}
The graded algebra maps $\iota_{g}$ and $\iota_{g,1}$ are stable isomorphisms.
\end{conjecture}

\end{document}